\documentclass[11pt]{article}

\usepackage{amsmath, amsthm, amssymb, amscd,paralist}
\usepackage{enumitem} 

\usepackage{graphicx}
\usepackage{bm}
\usepackage{tikz}
\usepackage[title]{appendix}
\usepackage{comment}

\DeclareGraphicsExtensions{.pdf,.eps,.png}
\graphicspath{{./Git/Figures/}}

\usepackage[linkbordercolor={0 0 1}]{hyperref}

\newtheorem{thm}{Theorem}[section]
\newtheorem{lem}[thm]{Lemma}
\newtheorem{prop}[thm]{Proposition}

\newtheorem{assump}{Assumption}

\begin{document}

\title{On frequentist coverage of Bayesian credible sets for estimation of the mean under constraints}

\author{Kevin Duisters\; and Johannes Schmidt-Hieber\footnote{The research leading to these results has received funding from the Dutch Science Organization (NWO) via a TOP II grant and a Vidi grant. Source code for the plots is available from \url{https://github.com/KevinDuisters/BayesCoverage} \newline Email: \href{mailto:k.l.w.duisters@math.leidenuniv.nl}{k.l.w.duisters@math.leidenuniv.nl}, \href{mailto:a.j.schmidt-hieber@utwente.nl}{a.j.schmidt-hieber@utwente.nl}}
 \vspace{0.1cm} \\
 {\em Leiden University and University of Twente} }

\maketitle
\date{}

\begin{abstract}
Frequentist coverage of $(1-\alpha)$-highest posterior density (HPD) credible sets is studied in a signal plus noise model under a large class of noise distributions. We consider a specific class of spike-and-slab prior distributions. Different regimes are identified and we derive closed form expressions for the $(1-\alpha)$-HPD on each of these regimes. Similar to the earlier work by \cite{MS2008}, it is shown that under suitable conditions, the frequentist coverage can drop to $1-3\alpha/2.$
\end{abstract}

%
%

\section{Introduction}
\label{sec.intro}

Despite the popularity in applied sciences of using Bayes for uncertainty quantification, frequentist properties of (Bayesian) credible sets remain poorly understood. In a few special cases, e.g. under conjugacy, it might be possible to identify the posterior as a distribution of some known class and to derive closed form expressions for the frequentist coverage of a given $(1-\alpha)$-credible set. Except for location problems, there is typically a gap with the frequentist coverage not matching the credibility (see \cite{MR95550} and the survey article \cite{MR2918001} including discussion \cite{zhang2011},\cite{wasserman2011}). Moreover, for parametric models one can argue via the Bernstein-von Mises theorem to show that a given $(1-\alpha)$-credible set is also an asymptotic $(1-\alpha)$-confidence set. In cases where the limiting shape of the posterior is complicated - and this comprises most of the nonparametric and high-dimensional models - very little can be said about the frequentist coverage of a credible set, see \cite{MR2906881, Castillo2013, Castillo2014, MR3357861, RSH2020, 2016arXiv160905067R} for some results and further references.

While smooth and slowly varying priors lead to Bernstein-von Mises type theorems even if the number of parameters grows with $o(n^{1/3})$ and $n$ the sample size (\cite{Panov2015}), the spiky structure of model selection priors in high-dimensional statistical models induces large biases for smallish parameter values. In general, the posterior converges then to a difficult mixture distribution over different candidate models (\cite{castillo2015}). Until now, almost nothing can be said about the frequentist coverage in such cases.

A spike-and-slab prior puts a fraction of the mass to zero to enforce sparsity of the posterior. It is conceivable that because of the strong prior belief the posterior is overconfident, resulting in rather small credible sets with low frequentist coverage. To test such claims, we study the simplest imaginable model, where we observe $X$ with 
\begin{align}
	X = \theta + \varepsilon,
	\label{eq.mod}
\end{align}
and $\varepsilon$ is drawn from a known distribution with c.d.f. $G$ and symmetric density $g=G'$ unimodal at zero. This guarantees that $\theta$ is the mean and the median of $X.$ The spike-and-slab prior (\cite{mitchellbeauchamp,JohnSilv04,george}) is of the form
\[\pi(\theta) \propto (1-w) \delta_0(\theta) + w \gamma(\theta), \quad  \theta \in \Theta\]
with $\delta_0$ the Dirac measure at zero, $\gamma$ a density and $w>0$ the mixing proportion. The simplest choice would be to take $\gamma$ as the improper uniform distribution on $\mathbb{R}.$ Here we study the slightly more general prior with (improper) slab distribution
\begin{align*}
	\gamma(\theta)=\mathbf{1}(|\theta|>\lambda).
\end{align*}

Throughout the paper we call the spike-and-slab prior with this slab distribution the {\it $\theta$-min prior.} The $\theta$-min prior would be a natural choice if we would know beforehand that the true $\theta$ is either zero or large. In the high-dimensional statistics literature, this is known as $\beta$-min condition \cite{MR2807761}. Increasing $\lambda$ forces the posterior to put more mass to zero and enhances posterior sparsity. The rationale is that if we observe a small value of $X,$ the posterior has to decide whether this has been generated from $\theta=0$ or $|\theta|>\lambda.$ The likelihood for the latter decreases as $\lambda$ increases, resulting in a larger fraction of posterior mass being assigned to zero. We believe that this property is an attractive feature in applications. Under large sample asymptotics it is possible to achieve any level of sparsity by fixing the mixing weight and increasing $\lambda.$ In contrast, for the traditional spike-and-slab as well as the horseshoe prior and its variations, the prior mass gets more and more concentrated around zero as the sample size increases. 

\begin{figure}[htbp!]
	\centering
	\includegraphics[trim = 0  0 0 0,scale=0.12]{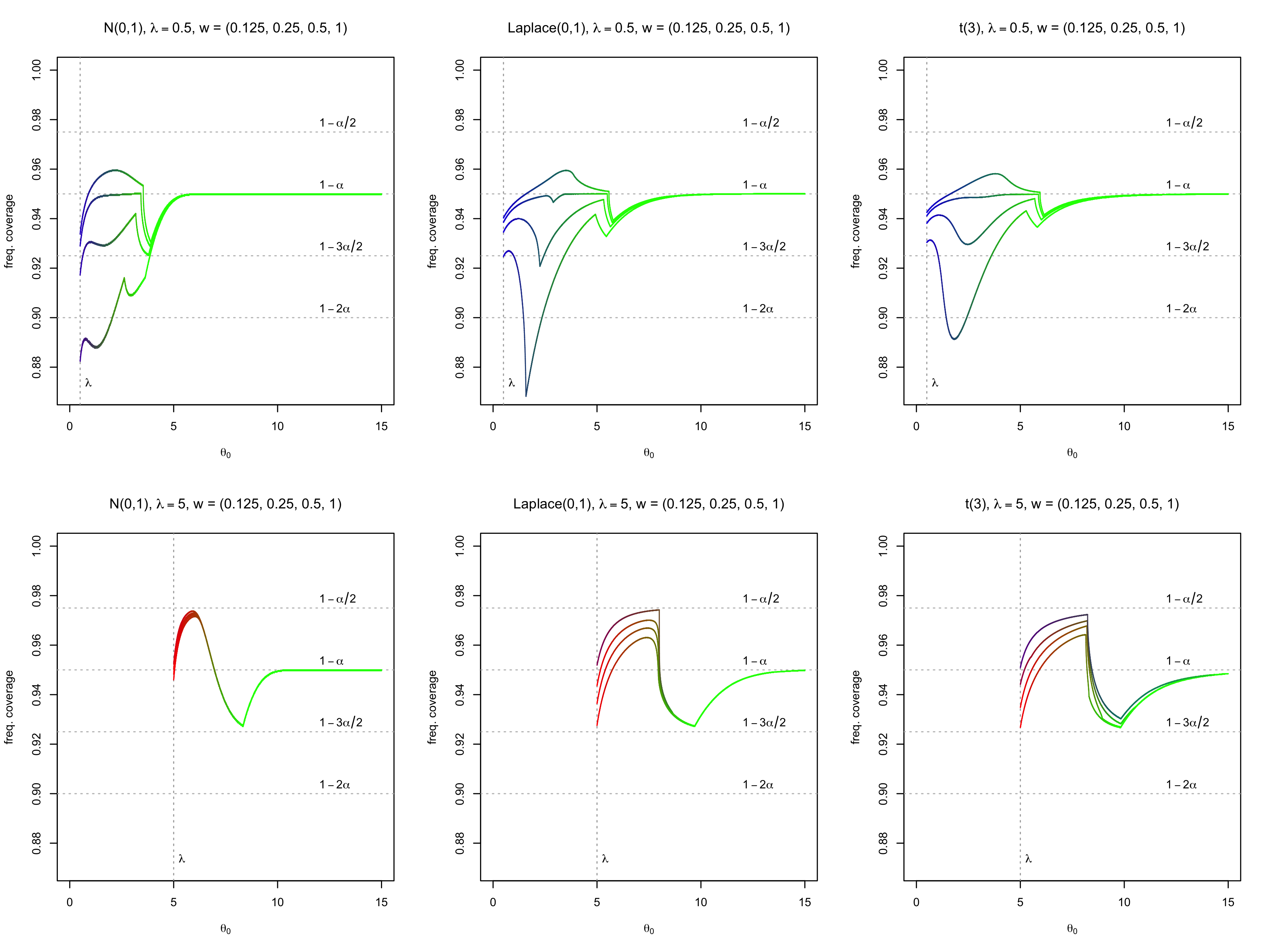}
	\caption{Frequentist coverage plots of $0.95$-highest posterior density credible set for $g$ standard normal (left column), Laplace (middle column), Student's $t(3)$ (right column) as well as $\lambda=0.5$ (top row) and $\lambda=5$ (bottom row). Each plot contains four lines generated by different $w$ values between $w=0.125$ (bottom line) to $w=1$ (top line). The coloring is explained in Section \ref{sec.coverage}.}
	\label{figcov}
\end{figure}

As common in Bayes, we denote by $\theta_0$ the value of $\theta$ that generated the data to distinguish it from the variable $\theta$ in the Bayes formula. We study the highest posterior density credible set (HPD). Figure \ref{figcov} displays the frequentist coverage of the HPD in dependence on $\theta_0$. Close to the threshold $\lambda,$ the frequentist coverage behaves quite erratic and rapid changes can occur. These sudden changes make the mathematical analysis highly non-trivial. In all cases there is a clearly visible local minimum with value between $1-3\alpha/2$ and $1-\alpha$ before the frequentist coverage increases to reach $1-\alpha$ for large values of $\theta_0.$ 

If $\lambda=0$ and $w=1,$ the prior is the uniform improper prior on the real line and one can directly verify that the frequentist coverage of the HPD for unimodal $g$ is exactly $1-\alpha.$ In this case, the coverage as a function of $\theta_0$ is flat. This shows that the properties of the frequentist coverage depend crucially on the choice of $\lambda$ and $w.$ To derive theoretical statements about the frequentist coverage, one of the challenges is to identify sufficient and necessary conditions on the parameters in the prior.

The first objective of this work is to derive closed-form expressions for the credible sets. It turns out that there are different regimes determining the behavior of the HPD credible set. In a first step, we identify these regimes and derive a closed-form expression for each of them. As a second step, it is shown how the expressions can be combined into a global formula for the credible sets that does not require knowledge of the regimes.

A starting point of our work is the beautiful theory developed in \cite{RoeWoodroofe2000, ZhangWoodroofe2003, MS2006, MS2008} studying variations of the model $X=\theta+\varepsilon$ with $\theta$ assumed to be non-negative. For this model, it is then natural to analyze the improper uniform prior on the positive half-axis, although other choices have been proposed as well \cite{DuanDunson2018}.  \cite{RoeWoodroofe2000} study the case $\epsilon \sim \mathcal{N}(0,\sigma^2)$ with variance $\sigma^2$ known.  \cite{ZhangWoodroofe2003} consider a variation with $\varepsilon \sim \mathcal{N}(0,\sigma^2)$ and unknown variance, assuming that we also observe $W\sim \sigma^2 \chi_r^2.$ For the improper prior $\pi(\theta, \sigma)=\sigma^{-1}\mathbf{1}(\sigma,\theta>0),$ it is shown that the posterior can be written as a truncated $t$-distribution. Based on this, an $(1-\alpha)$-credible set is derived for which it can be shown that the frequentist coverage is lower bounded by $(1-\alpha) / (1+\alpha)=1-2\alpha +O(\alpha^2).$ In \cite{MS2006} it is shown that this lower bound on the frequency coverage holds for a much larger class of problems. \cite{MS2008} studies the model $X=\theta+\varepsilon$ with error density $g$ assumed to be known and log-concave. Similar as in our analysis, regimes are identified on which the credible sets have different behavior. A complete characterization of the frequentist coverage is derived and in particular it is shown that the frequentist coverage is lower bounded by $1-3\alpha/2$ up to smaller order terms in $\alpha.$ This is sharper than the $1-2\alpha$ lower bound and it is shown that there is also one value of $\theta_0$ for which the lower bound is attained. In \cite{MR3066373} and \cite{MR3477727}, the analysis developed in \cite{MS2006} and \cite{MS2008} has been extended to a larger class of credible sets and allowing for some types of asymmetry in $g.$

The mathematical analysis of the two-sided setup considered here differs in some key aspects from the techniques developed in \cite{MS2008}. The HPD credible sets in \cite{MS2008} are intervals of the form $[L(X),U(X)]$ with $L$ and $U$ being non-decreasing functions in $X.$ For the analysis, this is a crucial property that does, however, not hold anymore in the two-sided setting considered here, see Lemma \ref{lem.Ua_decrease}. Frequentist coverage depends on the inverse of $L$ and $U.$ The fact that the inverse of $L$ and $U$ is set-valued and has a difficult structure is the major technical obstacle in our analysis and requires several new techniques to derive bounds for the frequentist coverage.

The paper is organized as follows. In Section \ref{sec.crediblesets}, we derive formulae for the HPD credible set. Results on the frequentist coverage of the credible sets are summarized and discussed in Section \ref{sec.crediblesets}. Post-selection is a frequentist method to determine confidence sets after a model selection procedure has been applied to the data. In Section \ref{sec.post_select} it is shown that in model \eqref{eq.mod}, posterior credible sets can be converted into post-selection sets. All proofs are deferred to Section \ref{sec.proofs}.

\section{Closed-form expressions for HPD credible set}
\label{sec.crediblesets}

We start by deriving an expression for the posterior. Throughout the article, $P_{\theta}$ denotes the distribution for $X=\theta+\varepsilon.$ It is convenient to define
\begin{align}
	\Delta_{\lambda}(x) := P_x\big(X \in [-\lambda, \lambda] \big) = G(\lambda - x) - G(-\lambda - x).
	\label{eq.delta_def}
\end{align}
Let $\delta_0$ be the point mass at zero. The posterior can be written as
\begin{align}
	\begin{split}
		\pi(\theta | X) &= \frac{g(\theta - X) \pi(\theta) }{\int_\Theta g(\theta - x) \pi(\theta) d\theta} \\
		&= \frac{(1-w)g(X)}{(1-w) g(X) + w (1-\Delta_{\lambda}(X))} \delta_0(\theta) + \frac{w g(X - \theta)\mathbf{1}(|\theta| > \lambda)}{(1-w) g(X) + w (1-\Delta_{\lambda}(X))}
	\end{split}
	\label{eq.posterior}
\end{align}
or in integrated form $\Pi(A|X) =\int_A \pi(\theta | X) \, d\theta.$ For $|\theta|>\lambda,$ we have $\pi(\theta|X) \propto g(X-\theta).$ If $g$ is strictly increasing on $(-\infty, 0)$ and strictly decreasing on $(0,\infty),$ $\pi$ is strictly increasing on $(-\infty, X)$ and strictly decreasing on $(X, \infty).$ 

The $(1-\alpha)$-HPD credible set given observation $X$ will be denoted by $\operatorname{HPD}_\alpha(X).$ If the point measure at zero does not cover yet $1-\alpha$ of the posterior mass, then $\operatorname{HPD}_\alpha(X)$ is the union of $\{0\}$ and the closed posterior level set $A_c \subseteq \mathbb{R}\setminus [-\lambda, \lambda]$ satisfying
\begin{align}
	\Pi( A_c \, | \, X ) = 1-\alpha - \Pi(0|X).
	\label{eq.HPD_as_level_set}
\end{align}

The $\theta$-min prior, the resulting posterior and the frequentist coverage are plotted in Figure \ref{figIntro}.

\begin{figure}[ht]
	\centering
	\includegraphics[trim = 0  0 0 0,scale=0.5]{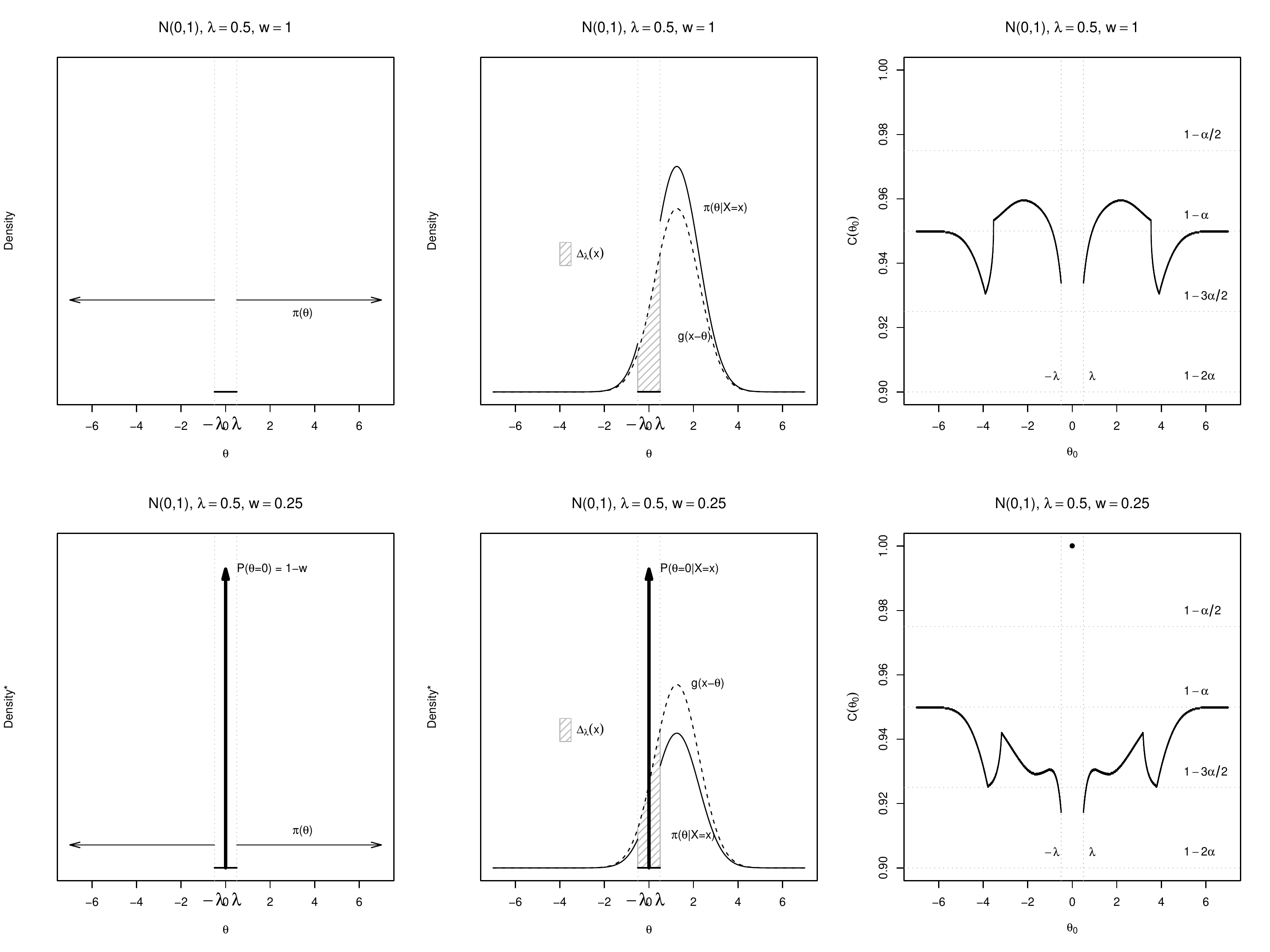}
	\caption{Illustration of $\theta$-min prior (left), posterior [solid] and likelihood [dashed] for $x=1.25$ (middle), and frequentist coverage of the 0.95-HPD credible set (right); using $g$ standard normal and $\lambda = 0.5$, with $w=1$ (top row) and $w=0.25$ (bottom row).} 
	\label{figIntro}
\end{figure}

For the further analysis, we always assume that the error distribution $g$ is in the class
\begin{align*}
	\mathcal{G}:=\big\{ g: &g \, \text{a positive and continuous density on} \, \mathbb{R}, \, \\
	&g \, \text{is symmetric around zero and strictly decreasing on} \, \mathbb{R}_+\big\}.
\end{align*}
This implies $g$ is unimodal with mode at zero. The c.d.f. of $g$ is denoted as $G : \mathbb{R} \mapsto (0,1)$ with corresponding inverse $G^{-1}(p) := \{ q \in \mathbb{R}: G(q) = p\}$ for any $p \in (0,1)$. We often use the symmetry induced properties $G(q)+G(-q)=1$ and $G^{-1}(p)+G^{-1}(1-p)=0.$

It might well happen that for some realizations of $X$ the $(1-\alpha)$-HPD credible set is the point measure at zero. The next result shows that if this is the case then there exists a unique solution $t_\alpha=t_\alpha(g, \lambda,w) \geq 0$ of 
\begin{align}
	\frac{w}{1-w} \frac{G(t_{\alpha}-\lambda)+G(-t_{\alpha}-\lambda)}{g(t_{\alpha})} = \frac{\alpha}{1-\alpha}.
	\label{eq.pt_mass_zero}
\end{align}

\begin{lem}[Posterior mass at 0]\label{lem.Ta}
	Let $g\in \mathcal{G}$. If there exists a non-negative solution $t_\alpha$ to \eqref{eq.pt_mass_zero} it is unique and $\Pi(0|X=x)\geq 1-\alpha$ if and only if $|x|\leq t_\alpha.$
\end{lem}

If the posterior mass at zero is strictly smaller than $1-\alpha,$ no solution to \eqref{eq.pt_mass_zero} exists and we set $t_\alpha:=-\infty$ in this case. If $|X|\leq t_\alpha,$ the HPD credible set consists of $\{0\}$ only. To derive closed-form expressions for the HPD credible sets, it remains to study $\{X: |X| >  t_\alpha\}.$ Set
\begin{align}
	R_1(x) &:= G^{-1}\Big( 1- \frac{\alpha}2 - \frac{1-w}{2w} \alpha g(x) - \frac{1-\alpha}2 \Delta_\lambda(x) \Big) \notag \\
	R_2(x) &:= G^{-1}\Big( 1- \alpha - \frac{1-w}{w} \alpha g(x) +\alpha \Delta_\lambda(x) +G(-\lambda-x) \Big) \label{eq.Rsdef} \\
	R_3(x) &:= G^{-1}\Big( 1- \frac{\alpha}2 - \frac{1-w}{2w} \alpha g(x) + \frac{\alpha}2 \Delta_\lambda(x) \Big), \notag 
\end{align}
with $G^{-1}(p):=+\infty$ for $p\geq 1$ and $G^{-1}(p):=-\infty$ for $p\leq -1.$ As $w$ becomes small, the arguments can become negative and $G^{-1}$ is otherwise not necessarily well-defined anymore for all $x.$ The next lemma shows on which domains $R_1,R_3$ are finite.

\begin{lem}\label{lem.well_def}
	It holds that $0<R_1(x) <R_3(x)< \infty$ for all $\{x: |x| >  t_\alpha\}.$
\end{lem}

By Lemma \ref{lem.tech.delta} (ii), $R_j(x)=R_j(-x)$ for $j\in\{1,3\}.$ To derive closed-form expressions of the HPD, set $$\mathcal{T}_{\alpha} =\{ x: |x|>t_\alpha\}$$ 
and define the following four regimes
\begin{align*}
	\mathcal{X}_{\textnormal{I}} &:= \big\{x\in \mathcal{T}_{\alpha} : |x| > \lambda + R_1(x) \big \}, \\
	\mathcal{X}_{\textnormal{II}} &:= \big\{x\in \mathcal{T}_{\alpha} :  -\lambda + R_3(x) < x \leq \lambda + R_1(x) \big \}, \\
	\mathcal{X}_{\textnormal{III}} &:= \big\{x\in \mathcal{T}_{\alpha} :   |x| \leq -\lambda + R_3(x)\big \}, \\
	\mathcal{X}_{\textnormal{IV}} &:= \big\{x\in \mathcal{T}_{\alpha} :  -x \in \mathcal{X}_{\textnormal{II}}  \big \}.
\end{align*}
By Lemma \ref{lem.well_def}, $R_1(x)$ and $R_3(x)$ are finite for all $x\in \mathcal{T}_{\alpha}$ and thus these sets are well-defined.  Lemma \ref{lem.partition}  shows that every $x\in \mathcal{T}_{\alpha}$ belongs to exactly one of these regimes, and Figure \ref{figR} provides an illustration of the functions $R_1$, $R_2$ and $R_3$.

\begin{lem}\label{lem.partition}
	The sets $\mathcal{X}_{\textnormal{I}}-\mathcal{X}_{\textnormal{IV}}$ form a partition of $\mathcal{T}_{\alpha}.$
\end{lem}

\begin{figure}[htbp!]
	\centering
	\includegraphics[trim = 0  0 0 0,scale=0.12]{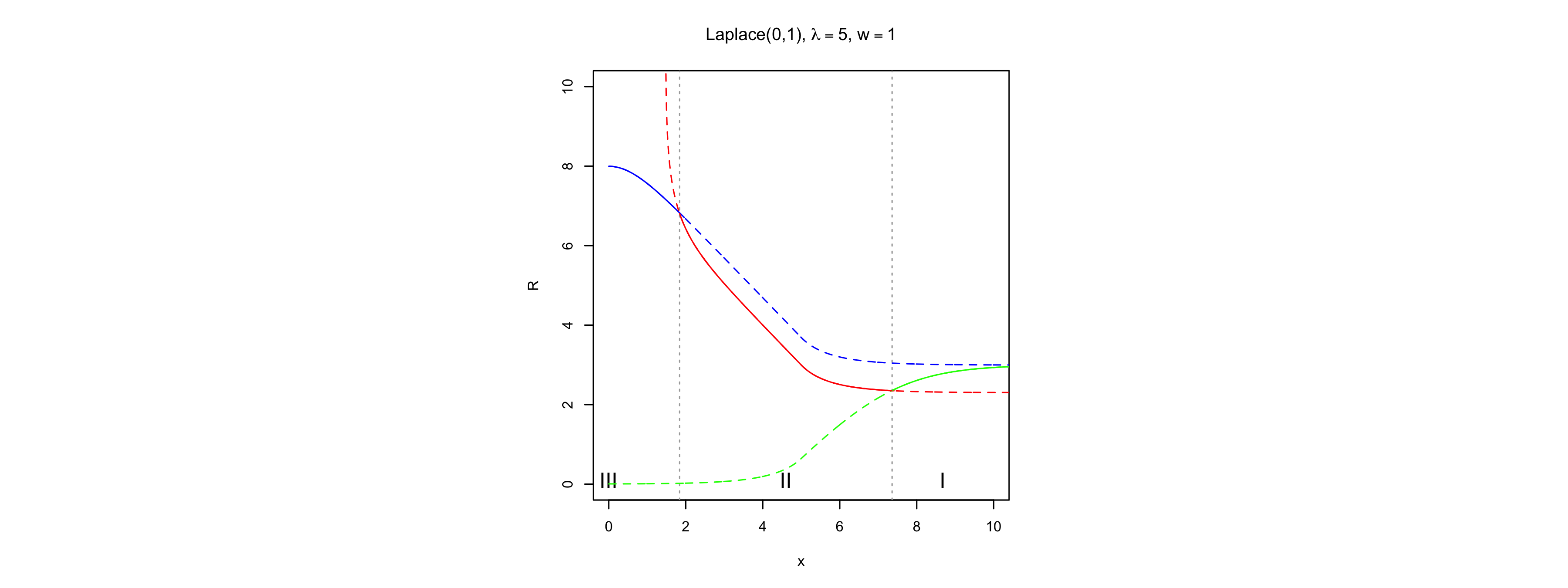}
	\caption{Functions $R_1$ (green), $R_2$ (red), and $R_3$ (blue) for the Laplace density $g$, $w=1$, $\alpha = 0.05$, and $\lambda = 5$. Due to symmetry around zero only the positive real line is displayed. On each regime (III, II, I), the 'active' function is highlighted by a solid line.}
	\label{figR}
\end{figure}

\noindent We can now state the first main result. 

\begin{thm}[Closed-form expression for credible set]\label{thm.LU}
	Let $g\in \mathcal{G}$ and $t_\alpha$ be as defined in Lemma \ref{lem.Ta}. Suppose $X=x$ is observed. If $|x| \leq t_{\alpha}$, $\operatorname{HPD}_{\alpha}(x) = \{0\}$. Otherwise, 
	\[ \operatorname{HPD}_{\alpha}(x) =\big( [  \operatorname{L}_{\alpha}(x), \operatorname{U}_{\alpha}(x)   ] \setminus (-\lambda, \lambda) \big) \cup (\{0\} \setminus \{1-w\})     \]
	with 
	\begin{equation*} 
		\operatorname{U}_{\alpha}(x)=
		\begin{cases}
			x +R_1(x) & x \in \mathcal{X}_{\textnormal{I}} \;, \\
			x + R_2(x)& x \in \mathcal{X}_{\textnormal{II}} \;,\\
			x + R_3(x)   & x \in \mathcal{X}_{\textnormal{III}} \;, \\
			-\lambda														   & x \in \mathcal{X}_{\textnormal{IV}} \; 
		\end{cases}
	\end{equation*} 
	and for any $x,$
	\begin{equation} 
		\operatorname{L}_{\alpha}(x) = -  \operatorname{U}_{\alpha}(-x).
		\label{eq.La_from_Ua}
	\end{equation} 
\end{thm}

The credible sets are therefore completely described by the functions $R_1,R_2,R_3.$ By studying the closed-form expressions of these functions, we see that the range of the credible set is influenced by $\alpha,$ a term involving the prior mass of the spike $1-w$ and a term depending on $\Delta_\lambda(x).$ This latter quantity is related to the mass of the distribution of $g(\cdot -X)$ on the interval $[-\lambda,\lambda].$ This mass is cut-out by the prior and redistributed by the Bayes formula on the remaining domain. The credible sets will be particularly difficult to analyze in regime $\mathcal{X}_{\textnormal{II}}$ as $R_2$ can decrease quickly forcing $\operatorname{U}_{\alpha}$ to decrease as well, see also Figure \ref{figR} and Figure \ref{figLU} (right).

The next lemma summarizes several elementary properties of the HPD credible sets. In particular, it shows that if the credible set for observing $X=0$ is not the point mass, none of the regimes is the empty set. 

\begin{lem}
	\label{lem.basic_props_Ualpha}
	If $g\in \mathcal{G},$ then,
	\begin{compactitem}
		\item[(i)] $\{x:|x|\geq \lambda+G^{-1}(1-\alpha/2)\} \subset \mathcal{X}_{\textnormal{I}}$
		\item[(ii)] If $\Pi(0|X=0) < 1 - \alpha,$ then $0 \in \mathcal{X}_{\textnormal{III}}.$
		\item[(iii)] If $\Pi(0|X=0) < 1 - \alpha,$ then $\mathcal{X}_{\textnormal{I}}-\mathcal{X}_{\text{IV}}$ are all non-empty. 
		\item[(iv)] For all sufficiently small slab weights $w,$ we have that $\mathcal{X}_{\textnormal{III}}$ is the empty set. 
		\item[(v)] For all sufficiently large $\lambda,$ $\mathcal{X}_{\textnormal{III}}$ is the empty set, $t_\alpha < \lambda$ and $(t_{\alpha}, \lambda] \subseteq \mathcal{X}_{\textnormal{II}}$.
	\end{compactitem}
\end{lem}

Part (i) states that all sufficiently large values of $x$ will be in regime $\mathcal{X}_{\textnormal{I}}.$ If the point mass at zero does not yet capture $1-\alpha$ of the posterior mass, then $\mathcal{X}_{\textnormal{III}}$ contains $0.$ By a modification of the proof for (ii), the assertion can be strengthened to $[-\eta,\eta] \subseteq \mathcal{X}_{\textnormal{III}}$ for a small $\eta>0.$ If, however, most of the prior mass is allocated for the spike, regime $\mathcal{X}_{\textnormal{III}}$ disappears. The same happens if $\lambda$ becomes large and $t_\alpha \in \mathcal{X}_{\textnormal{II}}.$

The next result provides us with an alternative formula for $\operatorname{U}_{\alpha}$ and $\operatorname{L}_{\alpha}.$ 
\begin{thm}
	\label{thm.altern_Ua}
	Let $g \in \mathcal{G}$ and define 
	\begin{align*}
		H_1(x) &:= x + \big(R_2(x)\wedge R_3(x)\big) \vee R_1(x), \quad 
		H_2(x) := -\lambda \wedge \big(x + R_1(x)\big).
	\end{align*}
	Then, for all $x$ with $|x|>t_\alpha,$
	\begin{align*}
		\operatorname{U}_{\alpha}(x) = H_1(x) \mathbf{1}\big(H_1(x) \geq \lambda\big) + H_2(x)  \mathbf{1}\big(H_1(x)  < \lambda\big).
	\end{align*}
\end{thm}

\begin{figure}[htbp!]
	\centering
	\includegraphics[trim = 0 0 0 0, scale=0.50]{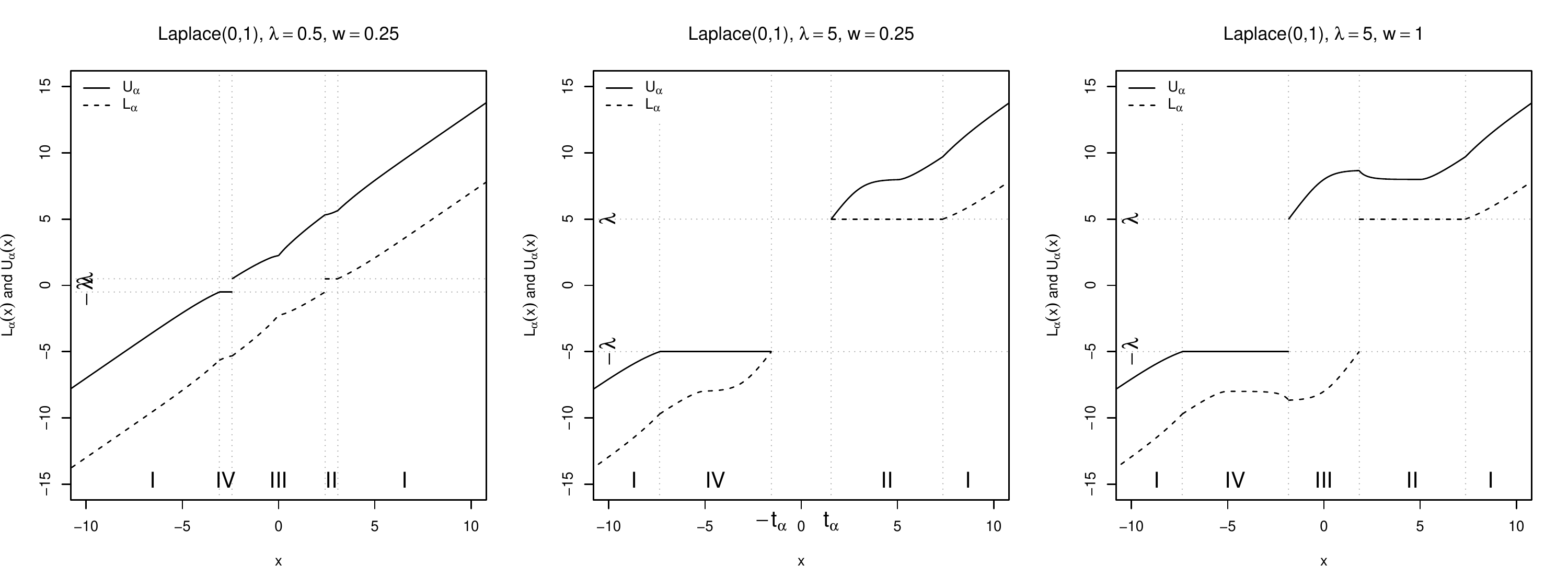}
	\caption{The functions $\operatorname{L}_{\alpha}$ (dashed) and $\operatorname{U}_{\alpha}$ (solid) with $\alpha = 0.05$ and $g$ the Laplace density; using $\lambda = 0.5$, $w = 0.25$ (left), $\lambda = 5$, $w=0.25$ (middle), and $\lambda = 5$, $w=1$ (right). }
	\label{figLU}
\end{figure} 

By Lemma \ref{lem.tech.delta} (iii), $g$ and $\Delta_\lambda$ are decreasing on $(0,\infty).$ Together with the definition of $R_1$ in \eqref{eq.Rsdef} this shows that

\begin{lem}\label{lem.mon}
	Let $g\in \mathcal{G}.$ Then, $R_1$ is an increasing function on $(0,\infty)$ and thus $\operatorname{U}_{\alpha}$ increases on $\mathcal{X}_{\textnormal{I}} \cap(t_{\alpha},\infty)$.
\end{lem}

The same is not always true in the other regimes as shown in Figure \ref{figLU} and Lemma \ref{lem.Ua_decrease}. This non-monotonic behavior in $\operatorname{U}_{\alpha}$ (and $\operatorname{L}_{\alpha}(x) = - \operatorname{U}_{\alpha}(-x)$) is the main obstacle in the subsequent theoretical analysis on frequentist coverage. 

\begin{lem}\label{lem.Ua_decrease}
	Let $g$ be the Laplace density. Assume $\alpha \in(0,\frac{1}{2})$, $w \in (\sqrt{2\alpha},1]$, and $\lambda \in ( \tfrac 12 \ln(\tfrac{2}{\alpha}) ,  \ln(\frac{1-\alpha}{\alpha} \frac{w}{1-w}))$. Then $\operatorname{U}_{\alpha}$ is strictly decreasing on $(\tfrac 12 \ln(\tfrac{2}{\alpha}),\lambda)$.
\end{lem}

The HPD consists of $\{0\}$ and one or two intervals, depending on the regime. We refer to the Lebesgue measure of the credible set as the "length" of the credible set. With the expression for the $\operatorname{HPD}_{\alpha}(X)$ in Theorem \ref{thm.LU}, it is easy to see that the length of the $(1-\alpha)$- HPD credible set is $2R_1(X),$ $R_2(X)+X-\lambda,$ $2R_3(X)-2\lambda$ and $R_2(-X) - X - \lambda$ in regime $\mathcal{X}_{\text{I}}-\mathcal{X}_{\text{IV}},$ respectively. In Figure \ref{figsize} it can be seen that the length of the credible sets can be considerably smaller than that of the nominal confidence interval.

\begin{figure}[htbp!]
	\centering
	\includegraphics[trim = 0 0 0 0, scale=0.50]{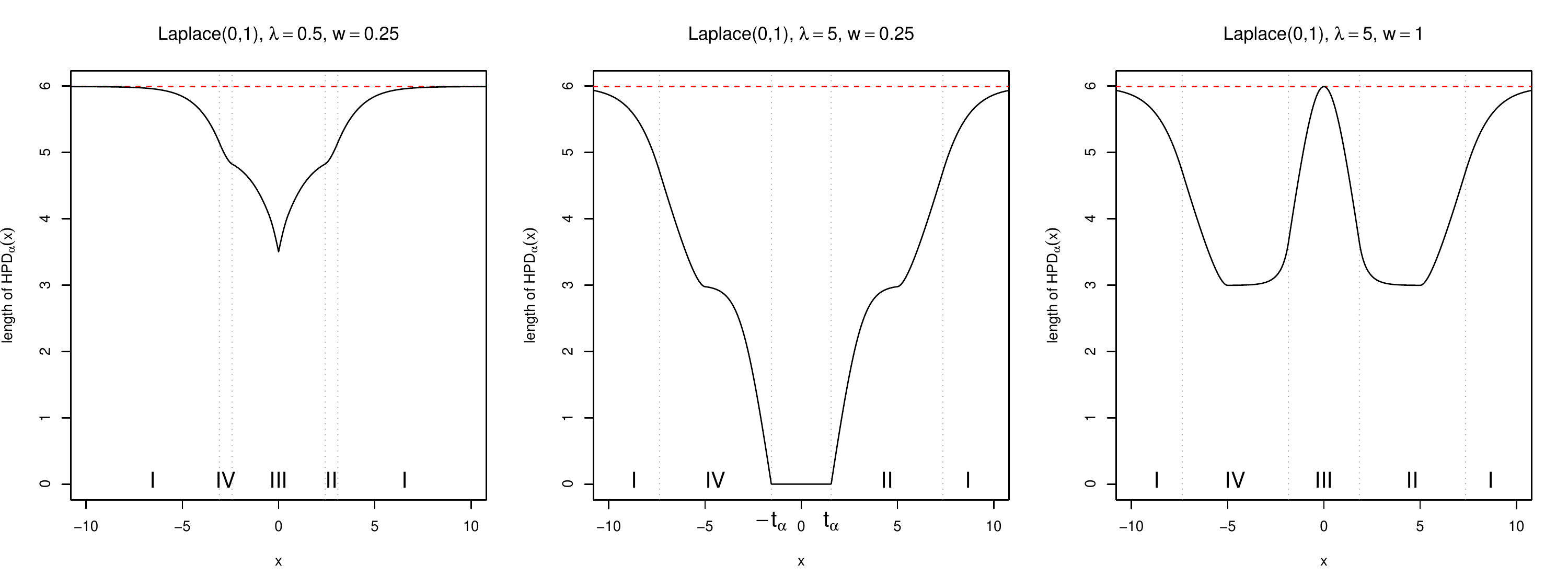}
	\caption{Length of $0.95$-HPD credible set as function of $x$ for the Laplace density $g$ and $\lambda = 0.5$, $w = 0.25$ (left), $\lambda = 5$, $w=0.25$ (middle), and $\lambda = 5$, $w=1$ (right). The length of the nominal confidence interval $2G^{-1}(1-\alpha/2)$ is plotted as dashed red line.}
	\label{figsize}
\end{figure}

\section{Coverage}\label{sec.coverage}
The frequentist coverage of the credible set is given by 
$$C(\theta_0) := \mathbb{P}_{\theta_0}\big(\theta_0 \in \operatorname{HPD}_{\alpha}(X)\big).$$
If $w<1,$ then $0\in \operatorname{HPD}_{\alpha}(X)$ for all $X$ and the frequentist coverage for $\theta_0=0$ is one. By construction the prior puts no mass on $[-\lambda,\lambda]\setminus\{0\}$ leading to zero coverage in this region. We study the interesting case $|\theta_0|>\lambda.$

Recall that for $|X|\leq t_\alpha$ the credible set is $\{0\}$ and for $|X|>t_\alpha$ the credible set is given by $ [\operatorname{L}_{\alpha}(X),\operatorname{U}_{\alpha}(X)]$ with $\operatorname{U}_{\alpha}, \operatorname{L}_{\alpha}$ as defined in Theorem \ref{thm.LU}. Thus, for $|\theta_0| > \lambda,$ the frequentist coverage  can be rewritten as
\begin{align}
	C(\theta_0) = \mathbb{P}_{\theta_0}\big(\theta_0 \in [\operatorname{L}_{\alpha}(X),\operatorname{U}_{\alpha}(X)] \cap \{|X| > t_{\alpha}\}\big).
	\label{eq.credible_rewritten}
\end{align}
Numerical simulations of $C(\theta_0)$ can be found in Figure \ref{figcov}. The behavior of the frequentist coverage $C(\theta_0)$ is determined by the regimes of the $x$ for which $\theta_0 \in [\operatorname{L}_{\alpha}(x),\operatorname{U}_{\alpha}(x)].$ This is displayed in Figure \ref{figcov} by line colors representing averages of the 'active' regimes with green, red and blue for regime $\mathcal{X}_{\textnormal{I}}, \mathcal{X}_{\textnormal{II}}$ and $\mathcal{X}_{\textnormal{III}},$ respectively.

The first result uses this formula to show that the frequentist coverage becomes smaller at any point $\theta_0$ as the spike at zero in the spike and slab prior gets more mass. While this is a qualitative result, we will later quantify the loss of frequentist coverage. 
\begin{lem}
	\label{lem.decrease_cov}
	Let $g \in \mathcal{G}.$ For any $|\theta_0|>\lambda,$ the frequentist coverage $C(\theta_0)$ decreases if the mixing weight on the slab prior distribution decreases.
\end{lem}

It is intuitively clear that the frequentist coverage should be symmetric in the sense that $C(\theta_0)=C(-\theta_0).$ To verify this, observe that $-X\sim \mathcal{N}(\theta_0,1)$ is equivalent to $X\sim \mathcal{N}(-\theta_0,1).$ By \eqref{eq.La_from_Ua},
\begin{align*}
	C(-\theta_0) &= \mathbb{P}_{-\theta_0}\big(-\theta_0 \in \operatorname{HPD}_{\alpha}(X)\big)
	= \mathbb{P}_{\theta_0}\big(-\theta_0 \in \operatorname{HPD}_{\alpha}(-X)\big)\\
	&= \mathbb{P}_{\theta_0}\big(\theta_0 \in \operatorname{HPD}_{\alpha}(X)\big) =C(\theta_0).
\end{align*}
Thus, from now on it will be enough to study the frequentist coverage for $\theta_0>\lambda.$ To better understand the frequentist coverage, we moreover define the lower and the upper coverage by
\begin{align}
	C^{-}(\theta_0) = \mathbb{P}_{\theta_0}\big(\theta_0 \in [\operatorname{L}_{\alpha}(X),X] \cap \{|X| > t_{\alpha}\}\big)
	\label{eq.coverage_lower}
\end{align}
and
\begin{align}
	C^{+}(\theta_0) = \mathbb{P}_{\theta_0}\big(\theta_0 \in (X, \operatorname{U}_{\alpha}(X)] \cap \{|X| > t_{\alpha}\}\big),
	\label{eq.coverage_upper}
\end{align}
respectively. It can happen that $\operatorname{L}_{\alpha}(x)>x.$ In this case $[\operatorname{L}_{\alpha}(x),x]$ is defined as the empty set. This is also the major complication in the proof of the following result.

\begin{lem}\label{lem.sum_cov}
	For $\theta_0>\lambda,$
	\begin{align*}
		C(\theta_0)=C^{-}(\theta_0)+C^{+}(\theta_0).
	\end{align*}
\end{lem}

In order to study frequentist coverage, it will be important to 'invert' $\operatorname{U}_{\alpha}$ and $\operatorname{L}_{\alpha}.$ If $\operatorname{L}_{\alpha}$ could be inverted, $\theta_0\in [\operatorname{L}_{\alpha}(X),X]$ would be the same as $X\in [\theta_0,\operatorname{L}_{\alpha}^{-1}(\theta_0)].$ Recall that $\operatorname{U}_{\alpha}, \operatorname{L}_{\alpha}$ are not necessarily monotone, see also Lemma \ref{lem.Ua_decrease} and Figure \ref{figLU}. The inverse functions 
\[\operatorname{U}_{\alpha}^{-1}(y)=\{x \in [-t_\alpha, t_\alpha]^c:\operatorname{U}_{\alpha}(x)=y\}  \quad \text{ and }  \quad  \operatorname{L}_{\alpha}^{-1}(y)=\{x\in [-t_\alpha, t_\alpha]^c:\operatorname{L}_{\alpha}(x)=y\}\]
are therefore set-valued. The next result seems rather obvious. The proof is, however, quite tricky, since $\operatorname{U}_{\alpha}$ and $\operatorname{L}_{\alpha}$ are not continuous everywhere.

\begin{lem}\label{lem.cov_reform}
	If $g \in \mathcal{G}$ and $\theta_0>t_\alpha,$ then the sets $\operatorname{L}_{\alpha}^{-1}(\theta_0)$ and $\operatorname{U}_{\alpha}^{-1}(\theta_0)$ are non-empty. Moreover,
	\begin{align*}
		\mathbb{P}_{\theta_0}\big(\theta_0\leq X\leq \inf \operatorname{L}_{\alpha}^{-1}(\theta_0)\big)\leq C^{-}(\theta_0) \leq \mathbb{P}_{\theta_0}\big(\theta_0\leq  X\leq \sup \operatorname{L}_{\alpha}^{-1}(\theta_0)\big)
	\end{align*}
	and 
	\begin{align*}
		\mathbb{P}_{\theta_0}\big(\big\{\sup \operatorname{U}_{\alpha}^{-1}(\theta_0)\leq X\leq \theta_0\big\}\cap \{|X|>t_\alpha\}\big)\leq C^+(\theta_0) \leq \mathbb{P}_{\theta_0}\big(\inf \operatorname{U}_{\alpha}^{-1}(\theta_0)\leq X\leq \theta_0\big).
	\end{align*}
\end{lem}

In particular, if we know that for instance $\inf \operatorname{L}_{\alpha}^{-1}(\theta_0)$ is in regime $\mathcal{X}_{\textnormal{I}},$ then, by Theorem \ref{thm.LU}, $\inf \operatorname{L}_{\alpha}^{-1}(\theta_0) - R_1(\inf \operatorname{L}_{\alpha}^{-1}(\theta_0) )=\theta_0$ and this yields the tractable lower bound
\begin{align}
	G\big(R_1(\inf \operatorname{L}_{\alpha}^{-1}(\theta_0))\big) - \frac{1}{2}
	= \mathbb{P}_{\theta_0}\big(\theta_0\leq X\leq \inf \operatorname{L}_{\alpha}^{-1}(\theta_0)\big)\leq C^{-}(\theta_0). 
	\label{eq.coverage_bds_specific}
\end{align}
The same reasoning can be applied to the other regimes and to derive upper bounds for $C^{-}(\theta_0).$

For given $X$ it is easy to compute $L_\alpha(X).$ If $\theta_0$ is known, it is less obvious how to compute elements in the set $\operatorname{L}_{\alpha}^{-1}(\theta_0).$ The next lemma provides an iterative formula for the smallest element in $\operatorname{L}_{\alpha}^{-1}(\theta_0).$

\begin{lem}\label{lem.compute}
	Let $g \in \mathcal{G}$ and $\theta_0> \lambda \vee t_\alpha$ be given. Let $a_0:=\theta_0$ and $a_{k+1}=\theta_0+R_1(a_k)$ for $k=0,\ldots.$ Then, $(a_k)_k$ converges to $\inf \operatorname{L}_{\alpha}^{-1}(\theta_0).$
\end{lem}

For $\theta_0>\lambda$ we have that $\operatorname{L}_{\alpha}^{-1}(\theta_0) \subset \mathcal{X}_{\textnormal{I}}.$ To see this, observe that if $x\in \operatorname{L}_{\alpha}^{-1}(\theta_0)$ then $\operatorname{L}_{\alpha}(x)> \lambda.$ Because of $\operatorname{L}_{\alpha}(x)=-\operatorname{U}_{\alpha}(-x),$ this is equivalent to $\operatorname{U}_{\alpha}(-x)<-\lambda.$ It then follows from the proof of Theorem \ref{thm.altern_Ua} that $x\in \mathcal{X}_{\textnormal{I}}.$ The set $\operatorname{U}_{\alpha}^{-1}(\theta_0)$ might spread over different regimes.

\begin{assump}\label{def.tail.decay}
	Assume $g\in \mathcal{G}$ and for some $\gamma >0$ and a constant $c_*,$
	\begin{align*} 
		G\Big(\frac 32 G^{-1}(t) \Big) < c_* t^{1+\gamma},  \quad \forall \, t\in (0,1), \quad \text{and} \quad g(x) \leq c_* \big( 1-G(x)\big)^\gamma, \quad \forall \, x\geq 0.
	\end{align*}
\end{assump}

The condition essentially requires exponential decay of $g.$ In the next result the condition is checked for a specific class of densities.

\begin{lem}\label{lem.exam}
	For any $\eta \in (0,1],$ consider the density $g(x) = c_\eta e^{-|x|^\eta}$ with normalization constant $c_\eta=1/\int e^{-|x|^\eta} \, dx.$ Then, Assumption \ref{def.tail.decay} holds for some finite $c_*=c_*(\gamma,\eta)$ and any $\gamma < (3/2)^\eta-1.$
\end{lem}

The considered class contains in particular the Laplace distribution. The result can also extended to $\eta>1$ for any $\gamma <1/2.$ In particular, it can be verified for the normal distribution and any $\gamma <1/2$ by following similar arguments and using Mill's ratio $\frac 12 < x(1-\Phi(x))/\phi(x) < 1$ for any $x\geq 1,$ see (9) in \cite{gordon1941}.

In \cite{MS2008}, log-concavity has been assumed. This means that the tails of the distribution cannot be heavier than Laplace. As shown in the previous lemma, in our setup, we can allow for tail decay $e^{-|x|^\eta}$ with small $\eta>0.$ This comes, however, at the price of larger remainder terms in the subsequent results.

The bounds on $C^{-}(\theta_0)$ and $C^+(\theta_0)$ derived below are of the form $1/2-\text{const.} \times \alpha$ plus terms of smaller order in $\alpha.$ Although precise constants can be obtained from the proofs, we find it more appealing to collect remainder terms using big O notation. Throughout the following, $O(\alpha^{1+\gamma})$ means that there is a term of the form $C\alpha^{1+\gamma}$ with $C$ a constant that can depend on the prior parameters $\lambda, w$ and the error density $g,$ but not on $\theta_0.$ In some statements, we assume that $\lambda$ is sufficiently large in comparison with $\alpha.$ For such statements $\lambda=\lambda(\alpha)$ is viewed as a function in $\alpha$ and $C$ will be independent of $\lambda.$

The level $\alpha$ is typically fixed in practice to a small value and to consider small $\alpha$ asymptotics might appear to be unsatisfactory. We believe that this asymptotics provides a good compromise between what is mathematically tractable and the behavior observed in simulations. The simulations in Figure \ref{figcov} also show that the remainder terms are indeed present and not an artifact of the proof as for instance the minimum is not attained exactly at $1-3\alpha/2.$

\begin{prop}
	\label{prop.upper_coverage}
	Suppose Assumption \ref{def.tail.decay} holds for some $c_*$ and $\gamma \in (0,1],$ then, we have
	\begin{compactitem}
		\item[(i)]  for $\theta_0>t_\alpha \vee \lambda,$
		\begin{align*}
			\frac{1}{2}-\alpha +O(\alpha^{1+\gamma})
			\leq C^{-}(\theta_0) \leq \frac{1-\alpha}{2},
		\end{align*}
		\item[(ii)] for $\theta_0>t_\alpha \vee \lambda$ and if $\sup \operatorname{U}_{\alpha}^{-1}(\theta_0)\geq \lambda,$
		\begin{align*}
			C^{-}(\theta_0)=\frac{1-\alpha}{2} +O(\alpha^{1+\gamma}),
		\end{align*}
		\item[(iii)] for $G(-\lambda)\leq \alpha$ and $t_\alpha \leq \lambda,$
		\begin{align*}
			\lim_{\theta_0 \downarrow \lambda} C^{-}(\theta_0) \leq \frac{1}{2}- \alpha + O(\alpha^{1+\gamma}).
		\end{align*}
	\end{compactitem}
\end{prop}

Thus, already if $\theta_0$ is slightly larger than $\lambda,$ the coverage $C^{-}(\theta_0)$ is $(1-\alpha)/2$ up to terms of the order $O(\alpha^{1+\gamma})$ by (ii). For $\theta_0$ approaching $\lambda$, the coverage $C^{-}(\theta_0)$ is, however, close to $1/2-\alpha$ as can be seen from combining (i) and (iii). Without the condition  $G(-\lambda)\leq \alpha$ in (iii), the assertion is not true. To see this, consider the extreme case with $\lambda=0$ and $w=1.$ In this case, the prior is the (improper) uniform prior on $\mathbb{R}$ and the credible sets are of the form $[X-G^{-1}(1-\alpha/2),X+G^{-1}(1-\alpha/2)].$ Consequently, $C^{-}(\theta_0)=(1-\alpha)/2$ for all $\theta_0$ and (iii) does not hold. Only for large $\lambda$ we will have the $1/2-\alpha$ coverage near $\lambda$ and to establish the result $G(-\lambda)\leq \alpha$ occurs very naturally.

In a next step, we collect some results on the behavior of $C^+(\theta_0).$

\begin{prop}\label{prop.Cpart2}
	Suppose Assumption \ref{def.tail.decay} holds for some $c_*$ and $\gamma \in (0,1].$
	\begin{compactitem}
		\item[(i)] If $\inf \operatorname{U}_{\alpha}^{-1}(\theta_0) \geq \lambda$ and $G(-2\lambda) \leq \tfrac{\alpha}2 \wedge (c_*\alpha^{1+\gamma}),$ then,
		\begin{align*}
			C^{+}(\theta_0) \leq \frac{1-\alpha}{2}+O\big(\alpha^{1+\gamma}\big).
		\end{align*}
		\item[(ii)] If $\sup \operatorname{U}_{\alpha}^{-1}(\theta_0) \geq \lambda+\tfrac 32 G^{-1}(1-\alpha/2)>t_\alpha,$ then,
		\begin{align*}
			C^{+}(\theta_0) 
			&\geq 
			\frac{1-\alpha}{2}+O\big(\alpha^{1+\gamma}\big).
		\end{align*}
		\item[(iii)] If $\lambda\geq t_\alpha$ and $G(-\lambda)\leq \alpha,$ then,
		\begin{align*}
			\inf_{\theta_0: \sup \operatorname{U}_{\alpha}^{-1}(\theta_0) \geq \lambda}C^{+}(\theta_0)&= 
			\frac{1}{2}- \alpha +O\big(\alpha^{1+\gamma}\big) + \alpha G\Big( \frac 12 G^{-1}(\alpha)\Big).
		\end{align*}
		\item[(iv)] If $t_\alpha=-\infty,$ then, there exists a $\theta_0^*>\lambda$ such that for any $\theta_0\in (\lambda, \theta_0^*],$
		\begin{align*}
			C^{+}(\theta_0) \geq \frac 12 - G(-\lambda).
		\end{align*}
		\item[(v)] If $t_\alpha> \theta_0>\lambda,$ we have $C^{+}(\theta_0)\leq G(-2\lambda).$
	\end{compactitem}
\end{prop}

For large values of $\theta_0,$ the coverage $C^{+}(\theta_0)$ is $(1-\alpha)/2$ up to smaller order terms in $\alpha.$ This is, however, not true for all values of $\theta_0.$ Part (iii) shows that there are some $\theta_0,$ for which $C^+(\theta_0)$ is approximately $\tfrac 12 -\alpha.$ The conditions $\lambda\geq t_\alpha$ and $G(-\lambda)\leq \alpha$ in (iii) are satisfied for instance for large $\lambda,$ see Lemma \ref{lem.basic_props_Ualpha} (v). The last two statements show that near $\lambda,$ $C^{+}(\theta_0)$ can be close to $0$ or $1/2$ for large $\lambda.$

The next result is a consequence of Lemma \ref{lem.sum_cov}, Proposition \ref{prop.upper_coverage} and Proposition \ref{prop.Cpart2}.

\begin{thm}\label{thm.main}
	Suppose Assumption \ref{def.tail.decay} holds for some $c_*$ and $\gamma \in (0,1].$
	\begin{itemize}
		\item[(i)] If $\alpha \leq (4c_*)^{-1/\gamma}\wedge (1/2),$ $\lambda \geq t_\alpha,$ $G(-2\lambda) \leq c_*\alpha^{1+\gamma},$ and $\sup \operatorname{U}_{\alpha}^{-1}(\theta_0) \geq \lambda+\tfrac 32 G^{-1}(1-\alpha/2),$ then,
		\begin{align*}
			C(\theta_0) = 1-\alpha+O\big(\alpha^{1+\gamma}\big).
		\end{align*}
		\item[(ii)] If $\theta_0>\lambda\geq t_\alpha$ and $G(-\lambda)\leq \alpha,$
		\begin{align*}
			\inf_{\theta_0: \sup \operatorname{U}_{\alpha}^{-1}(\theta_0) \geq \lambda}C(\theta_0)&= 
			1-\frac 32 \alpha +O\big(\alpha^{1+\gamma}\big) + \alpha G\Big( \frac 12 G^{-1}(\alpha)\Big).
		\end{align*}
	\end{itemize}
\end{thm}

This means that for large values of $\theta_0,$ the $(1-\alpha)$-HPD credible set is nearly a $(1-\alpha)$-confidence set. For values of $\theta_0$ below $\{\theta_0: \sup \operatorname{U}_{\alpha}^{-1}(\theta_0) \geq \lambda\},$ the situation is much more complex and strongly depends on the interplay between $t_\alpha, \lambda, w,$ see Figure \ref{figcov}.

It is instructive to compare the frequentist coverage to the earlier work by Marchand and Strawderman \cite{MS2008} on lower-bounded mean problems. By a shift in the parameter space, their prior is $\pi_{\lambda}^{\operatorname{MS}}(\theta) \propto \mathbf{1}(\theta > \lambda),$ while the $\theta$-min prior for $w=1$ is $\pi_{\lambda}(\theta) \propto \mathbf{1}(|\theta|>\lambda)$.

\begin{thm}
	\label{thm.coverage_inclusion}
	Let $g \in \mathcal{G},$ $w=1,$ $\alpha \in (0,1), \theta_0>\lambda>0.$ Denote the frequentist coverage of the $\operatorname{HPD}_{\alpha}$ stemming from prior $\pi^{\operatorname{MS}}_{\lambda}$ by $C^{\operatorname{MS}}(\theta_0)$, and that of prior $\pi_{\lambda}$ by $C(\theta_0)$ as before, then, for $\theta_0>\lambda,$
	\begin{align}
		C^{\operatorname{MS}}(\theta_0) &< C(\theta_0)+G(-\theta_0).
		\label{eq.cov_inc_1}
	\end{align}
\end{thm}

In particular, we always have $G(-\theta_0)\leq G(-\lambda).$ The additional term $G(-\theta_0)$ thus disappears if $\lambda$ gets large. However, for $w=1$ and $\lambda$ small, we expect $C(\theta_0) \approx 1-\alpha$ for all $\theta_0,$ as discussed in Section \ref{sec.intro} for the limit $\lambda=0.$ At the same time, $C^{\operatorname{MS}}(\theta_0)$ can reach $1-\alpha/2$ up to smaller order terms in $\alpha$ as shown for the Laplace error density in Example 3 of \cite{MS2008} (for $\theta_0 = -\ln(\alpha)$ one has $C^{\operatorname{MS}}(\theta_0) = 1 - \alpha/2 - O(\alpha^2)$). This shows that \eqref{eq.cov_inc_1} cannot hold without an additional term on the right hand side.

\section{Relation to post-selection sets}
\label{sec.post_select}

Consider the $\theta$-min prior with all mass on the slab distribution, that is, $w=1.$ In this section we show that for the model $X=\theta+\varepsilon$ there is a duality between posterior credible sets under this prior and post-selection sets. In particular, it is possible to derive post-selection sets from credible sets and vice versa. 

In high-dimensional statistics, it is natural to first identify some relevant variables using a variable selection method such as the LASSO. Given a variable that is selected by the method, we then want to construct a $(1-\alpha)$-confidence interval. Such procedures are also known as post-selection methods \cite{lee2016}. The issue with post-selection is that the selection step is already data dependent making it highly non-trivial to construct a valid confidence set as a second step. 

Post-selection should be naturally applied to high-dimensional problems. However, it is instructive to study it for the model $X=\theta+\varepsilon,$ see also \cite{Weinstein2010}. Shrinkage based methods, such as the LASSO, select $X$ in this model if $|X|> \lambda$ for $\lambda$ a known threshold. A $(1-\alpha)$-post-selection set is of the form $\operatorname{PS}_{\alpha}(X),$ such that
\begin{align}
	P_{\theta_0}\big( \theta_0 \in \operatorname{PS}_{\alpha}(X) \, \big| \, |X| \geq \lambda\big) \geq 1-\alpha.
	\label{eq.PS}
\end{align}
Compared to the credible sets before the role of $\theta_0$ and $X$ are interchanged and as shown next, we can in this case obtain a $(1-\alpha)$-post-selection set by 'inverting' any $(1-\alpha)$-credible set. As before, we define the inverse of a set valued function $A$ as $A^{-1}(x):=\{y:x\in A(y)\}.$

\begin{lem}
	For $\lambda >0,$ let $\operatorname{CS}_{\alpha}(X)$ be a $(1-\alpha)$-credible set for the $\theta$-min prior with $w=1.$ Then, $\operatorname{PS}_\alpha(X)=\operatorname{CS}_{\alpha}^{-1}(X)$ is a $(1-\alpha)$-post selection set satisfying \eqref{eq.PS}.
\end{lem}

\begin{proof}
	The distribution of $X \big| |X| \geq \lambda$ is 
	\begin{align*}
		P_{\theta_0}\big( X \in A \, \big| \, |X| \geq \lambda\big) = \frac{\int_A g(x-\theta_0) dx}{\int_{|x|\geq \lambda} g(x-\theta_0) dx}, \quad \text{for all measurable} \ A \subseteq \mathbb{R}\setminus [-\lambda, \lambda].
	\end{align*}
	This should be compared to the posterior distribution for the $\theta$-min prior with $w=1,$ given by
	\begin{align*}
		\Pi\big(A \big|X\big) = \frac{\int_A g(\theta-X) d\theta}{\int_{|\theta|\geq \lambda} g(\theta-X) d\theta}, \quad \text{for all measurable} \ A \subseteq \mathbb{R}\setminus [-\lambda, \lambda].
	\end{align*}
	Using the formula and the fact that by definition of an $(1-\alpha)$-credible set, $\Pi(\operatorname{CS}_{\alpha}(X)|X)\geq 1-\alpha,$ we must have that $P_{\theta_0}( X \in \operatorname{CS}_{\alpha}(\theta_0) \, | \, |X| \geq \lambda)\geq 1-\alpha.$ Consequently, $\operatorname{PS}_\alpha(X)=\operatorname{CS}_{\alpha}^{-1}(X)$ is a $(1-\alpha)$-post selection set.
\end{proof}

As shown for the $(1-\alpha)$-HPD credible set, the frequentist coverage fluctuates around $1-\alpha.$ The previous result shows that a more natural comparison would be to relate credible sets to the frequentist coverage under the conditional distribution $X\big| |X|\geq \lambda.$ For a related argument in the case of lower bounded means, see part A of Section V in \cite{RoeWoodroofe2000}. While \cite{lee2016} deals with post selection in the linear regression model, the mathematical analysis has some striking similarities with our proofs to establish bounds on the frequentist coverage. If the connections between Bayes and post-selection can be extended to more complex models, this might open a new route to compute valid post-selection sets. Another interesting direction are the Bayesian post-selection sets discussed in \cite{Yekutieli2012}.

\section{Discussion}

Several natural extensions remain to be explored. One of the rather restrictive assumptions is the improper uniform prior distribution on the slab. Characterization of the frequentist coverage for more general classes of spike-and-slab priors or the horseshoe and its variants \cite{MR2650751} are not straightforward and likely require new proof strategies. Another direction is to consider more general models with natural extensions being the sequence model and the high-dimensional linear regression model. A major challenge is to unify Bayesian and frequentist uncertainty quantification by constructing sets that are simultaneous $(1-\alpha)$-credible sets and $(1-\alpha)$-confidence sets.

\section{Proofs}
\label{sec.proofs}

\subsection{Proofs for Section \ref{sec.crediblesets}}

Basic properties of $\Delta_{\lambda}(x)$ are summarized in the next lemma. 

\begin{lem}\label{lem.tech.delta}
	Let $g\in \mathcal{G},$ $\lambda > 0$ and recall that $\Delta_{\lambda}(x) = G(\lambda - x) - G(-\lambda - x)$. We have that
	\begin{compactitem}
		\item[(i)] $\Delta_{\lambda}(x) > 0$ for any $x \in \mathbb{R},$
		\item[(ii)] $\Delta_{\lambda}(x)=\Delta_{\lambda}(-x)$ for all $x,$
		\item[(iii)]$\Delta_{\lambda}(x)$ is strictly increasing on $(-\infty,0)$ and strictly decreasing on $(0,\infty),$
		\item[(iv)] $\Delta_{\lambda}(x)$ maximal for $x=0$, with $\Delta_{\lambda}(0) = 1 - 2G(-\lambda) < 1.$
	\end{compactitem}
\end{lem}
\begin{proof}
	{\it (i):} Since $\lambda>0$ and $G$ is strictly monotone, $G(\lambda - x) > G(-\lambda - x).$ {\it (ii):} The symmetry of $g$ implies that $G(q)=1-G(-q)$ for any real $q.$ Therefore,  $\Delta_{\lambda}(-x) := G(\lambda + x) - G(-\lambda + x) = 1-G(-\lambda-x) - 1 + G(\lambda-x) =\Delta(x).$ {\it (iii):} Let $x>0.$ Since $g$ is strictly increasing on $(-\infty,0)$ and strictly decreasing on $(0,\infty)$, the derivative $\Delta_{\lambda}'(x) = g(-\lambda - x) - g(\lambda - x)$ is negative iff $ | -\lambda - x| = |\lambda + x| > |\lambda - x|$ which in turn is equivalent to $x>0.$ For $x<0,$ the result follows from the first part and {\it (ii).} {\it (iv):} Follows from $(iii).$
\end{proof}

\begin{proof}[Proof of Lemma \ref{lem.Ta}]
	Set $h(x):=\Pi(0|X=x).$ We first show that if there exists a non-negative solution in $x$ to $h(x)=1-\alpha,$ this is unique. We have $h(x) = 1/[1+w(1-\Delta_\lambda(x))/((1-w)g(x)) ].$ Notice that Lemma \ref{lem.tech.delta} (ii) and $g\in \mathcal{G},$ $g(x)=g(-x)$ imply $h(x)=h(-x).$ Moreover, $h$ is strictly decreasing on $[0,\infty),$ since by Lemma \ref{lem.tech.delta} (iii), $1/g(x)$ and $1-\Delta_\lambda(x)$ are both strictly increasing on this interval. This shows that any non-negative solution to $h(x)=1-\alpha$ must be unique.

	We now show that $h(t_\alpha)=1-\alpha$ if and only if $t_\alpha$ is a solution to \eqref{eq.pt_mass_zero}. By symmetry of $g$ and \eqref{eq.delta_def}, it follows that $1-\Delta_{\lambda}(x) = 1- (G(\lambda - x) - G(-\lambda - x))= G(x-\lambda)+G(-\lambda-x).$ If $w<1$,
	\begin{align}
		h(x) = \Big(1 + w \frac{G(x-\lambda)+G(-\lambda-x)}{(1-w)g(x)} \Big)^{-1}.
		\label{eq.post_point_mass}
	\end{align}
	Rewriting this expression shows that $h(x)= 1-\alpha$ if and only if \eqref{eq.pt_mass_zero} holds. As shown before $h$ is strictly decreasing and symmetric and so we must have $h(x)\geq 1-\alpha$ if and only if $|x| \leq t_{\alpha}$.
\end{proof}

\begin{proof}[Proof of Lemma \ref{lem.well_def}]
	Define $B_i(x)$ through the equation $R_i(x) = G^{-1}(B_i(x)),$ $i=1,2,3.$ To show the first part of the statement, we need to verify that $1/2<B_1(x),B_3(x)<1$ for all $x$ with $|x|>t_\alpha$. Observe that due to $\Delta_\lambda(x) \geq 0,$ we have that $B_1(x) \leq B_3(x)$ for all $x.$ By definition as a probability also $\Delta_\lambda(x) < 1.$ This shows that $B_1(x) \leq B_3(x)<1.$ From the proof of Lemma \ref{lem.Ta}, we have that $|x|>t_\alpha$ if and only if $h(x)<1-\alpha.$ Rewriting this using \eqref{eq.post_point_mass} shows that then also
	\begin{align}
		- \frac{1-w}{2w}\alpha g(x) >  -\frac{1-\alpha}2 \big(1-\Delta_\lambda(x)\big).
		\label{eq.well_def_eq1}
	\end{align}
	With \eqref{eq.Rsdef}, this yields $B_1(x)> 1/2$ for all $x$ with $|x|>t_\alpha.$ Therefore also $0<R_1(x)\leq R_3(x) <\infty$ on $\{x: |x| >  t_\alpha\}$ proving the first part of the claim.
\end{proof} 

Denote by $\operatorname{sign}(x) = \mathbf{1}(x>0)-\mathbf{1}(x<0)$ the sign function. We also set $\operatorname{sign}(+\infty)=1$ and $\operatorname{sign}(-\infty)=-1.$ Given two functions $f,h$  defined on $\mathbb{R},$ we say that $f$ and $h$ are sign equivalent if $\operatorname{sign}(f(x))=\operatorname{sign}(h(x))$ for all $x\in \mathbb{R}.$

\begin{lem}[Sign equivalence] \label{lem.sign.equiv}
	\begin{compactitem}
		\item[(i)] The functions $R_2(x)-R_1(x)$ and $\lambda+R_1(x)-x$ are sign equivalent.
		\item[(ii)] The functions $R_2(x) - R_3(x)$ and $R_3(x)- x-\lambda$ are sign equivalent.
	\end{compactitem}
	
\end{lem}
\begin{proof}
	{\it (i):} We show that $R_2(x)-R_1(x)<0$ implies  $\lambda+R_1(x)-x<0.$ Define
	\begin{align}
		D(X) := \frac{1-w}w  g(X) + 1-\Delta_{\lambda}(X),
		\label{eq.D_def}
	\end{align}
	for the rescaled denominator in the Bayes formula and observe that 
	\begin{align}
		R_1(X) = G^{-1}\Big(\frac 12 + \frac{1-\alpha -\Pi(0|X)}{2} D(X) \Big).
		\label{eq.R1_identity}
	\end{align}
	Arguing similarly as for \eqref{eq.R1_identity} and using $\Delta_\lambda(X)=G(\lambda-X)-G(-\lambda-X),$
	\begin{align}
		R_2(X) = G^{-1}\Big(G(\lambda-X) + \big(1-\alpha -\Pi(0|X)\big) D(X) \Big).
		\label{eq.R3_identity}
	\end{align}
	Recall that $G^{-1}$ is strictly increasing. 	From \eqref{eq.R1_identity} and \eqref{eq.R3_identity}, we have $R_2(x)-R_1(x)<0$ if and only if 
	\begin{align*}
		\frac 12 > G(\lambda-x) + \frac{1-\alpha -\Pi(0|X=x)}{2}D(x)
	\end{align*}
	which is equivalent to 
	\begin{align*}
		-R_1(x) &= - G^{-1}\Big( \frac{1}{2}+ \frac{1-\alpha -\Pi(0|X=x)}{2} D(x) \Big) \\
		&= G^{-1}\Big( \frac{1}{2}- \frac{1-\alpha -\Pi(0|X=x)}{2} D(x) \Big) \\
		&> \lambda -x,
	\end{align*}
	where in the first equation we use \eqref{eq.R1_identity} and for the second equation the fact that $G^{-1}(1/2-z)=-G^{-1}(1/2+z)$ for all real $z.$ The inequality can be rewritten as $\lambda+R_1(x)-x <0.$ By following the same arguments one can also show that $R_2(x)-R_1(x)=0$ implies  $\lambda+R_1(x)-x=0$ and $R_2(x)-R_1(x)>0$ implies  $\lambda+R_1(x)-x>0.$ This completes the proof for $(i).$
	
	{\it (ii):} We show that $R_3(x) - x -\lambda>0$ implies $R_2(x)-R_3(x)>0.$ Rewriting  $R_3(x) - x -\lambda>0$ and using the definition of $R_2$ and $R_3$ in \eqref{eq.Rsdef}, 
	\begin{align*}
		G(x+\lambda) < 1- \frac{\alpha}{2} - \frac{1-w}{2w}\alpha g(x) + \frac{\alpha}{2}\Delta_\lambda(x)
	\end{align*}
	and therefore
	\begin{align*}
		R_2(x) &= G^{-1}\Big( 1 -\alpha - \frac{1-w}{w}\alpha g(x) +\alpha \Delta_\lambda(x) +G(-\lambda-x)\Big)\\
		&> G^{-1}\Big( G(x+\lambda) -\frac{\alpha}2 - \frac{1-w}{2w}\alpha g(x) +\frac{\alpha}2 \Delta_\lambda(x) +G(-\lambda-x)\Big)\\
		&= R_3(x),
	\end{align*}
	using that $G(x+\lambda)+G(-x-\lambda)=1$ for the last step.  Hence $R_2(x) >R_3(x).$ The other parts of $(ii)$ follow by the same arguments.
\end{proof}

\begin{proof}[Proof of Lemma \ref{lem.partition}]
	Observe that a fixed $x^* \in \mathcal{T}_{\alpha}$ determines the values $R_1(x^*)$ and $R_3(x^*).$ If $-\lambda + R_3(x^*) <\lambda + R_1(x^*),$ then there exists exactly one regime containing $|x^*|$ proving the result for this case.  If $-\lambda + R_3(x^*) \geq \lambda + R_1(x^*),$ then, $\mathcal{X}_{\textnormal{II}}$ is empty but it could well happen that $|x^*|\in \mathcal{X}_{\textnormal{I}} \cap \mathcal{X}_{\textnormal{III}}$ if $-\lambda + R_3(x^*) \geq |x^*| > \lambda + R_1(x^*).$ Suppose this is possible. Recall that $R_1(x^*)=R_1(|x^*|)$ and $R_3(x^*)=R_3(|x^*|).$ By Lemma \ref{lem.sign.equiv} it follows then that $R_2(|x^*|)<R_1(|x^*|)$ and $R_2(|x^*|)\geq R_3(|x^*|).$ Hence $R_1(|x^*|)>R_3(|x^*|).$ This is, however, a contradiction to the definition in \eqref{eq.Rsdef} implying $R_1(x)\leq R_3(x)$ for all $x.$
\end{proof}

\begin{proof}[Proof of Theorem \ref{thm.LU}]	
	We only discuss the case $0<w<1.$ For $w=1$ the result can be obtained by following the same arguments. Since the point mass of the posterior at $0$ is contained in the HDP and $|x| > t_{\alpha},$
	\begin{align*}  
		\Pi\big(\operatorname{HPD}_{\alpha}(X)\setminus\{0\} \big| X \big)  & = 1-\alpha    -\Pi(0|X)  >0 \; .
	\end{align*}
	By \eqref{eq.HPD_as_level_set} it is sufficient to construct a posterior level set that contains $1-\alpha -\Pi(0|X)$ of the posterior mass. 
	
	(i) Suppose $X \in \mathcal{X}_{\text{I}}$ and $X$ positive, the result follows similarly for $X$ negative. It must hold that $X>\lambda.$ Due to the assumptions on $g,$ the posterior density is centered at $X.$ The posterior density is symmetric around $X,$ in the sense that for $a \leq X- \lambda,$ $\pi(X-a |X)=\pi(X+a|X).$
	
	Consider now the interval $[L_\alpha(X),U_\alpha(X)]=[X - R_1(X), X + R_1(X)].$ We show that this interval has posterior probability $1-\alpha    -\Pi(0|X).$ Since $R_1(X)< X-\lambda$ by definition of $\mathcal{X}_{\text{I}},$ it follows that $[\operatorname{L}_{\alpha}(X),\operatorname{U}_{\alpha}(X)] \cup\{0\}$ is the unique $(1-\alpha)$-HPD credible set. 
	Since $0 \not \in [\operatorname{L}_{\alpha}(X),\operatorname{U}_{\alpha}(X)]$ in this regime, using the definition of $D(X)$ given in \eqref{eq.D_def} and the representation of $R_1$ in \eqref{eq.R1_identity}, we have
	\begin{align*} 
		\Pi\big([\operatorname{L}_{\alpha}(X),\operatorname{U}_{\alpha}(X)] \big | X \big)  &= \frac{G(\operatorname{U}_{\alpha}(X) - X) - G(\operatorname{L}_{\alpha}(X) - X)}{D(X)}
		= \frac{G(R_1(X)) - G(- R_1(X))}{D(X)}\\
		&= \frac{2G(R_1(X)) - 1}{D(X)} = 1-\alpha -\Pi(0|X).
	\end{align*}
	The result for $X \in \mathcal{X}_{\text{I}}$ is obtained since $- \operatorname{U}_{\alpha}(-X) = X -R_1(-X) =X-R_1(X) =\operatorname{L}_{\alpha}(X).$

	(ii) Suppose $X \in \mathcal{X}_{\text{II}}.$ Since $-X \in \mathcal{X}_{\text{IV}},$ we have $\operatorname{L}_{\alpha}(X)=-\operatorname{U}_{\alpha}(-X) =\lambda.$ Using \eqref{eq.R3_identity},
	\begin{align}
		\Pi\big([\operatorname{L}_{\alpha}(X),\operatorname{U}_{\alpha}(X)] \big | X \big) = \frac{G(R_2(X))-G(\lambda-X)}{D(X)} = 1-\alpha -\Pi(0|X).
		\label{eq.R2_satisfies}
	\end{align}
	It remains to show that this is a level set. Since the posterior has zero mass on $(-\lambda,\lambda)\setminus\{0\},$ this is the same as saying that the posterior density at $-\lambda$ is strictly smaller than the posterior density at $\operatorname{U}_{\alpha}(X)$ or equivalently, 
	\begin{align*}
		g(X-(-\lambda)) < g(X-\operatorname{U}_{\alpha}(X)).
	\end{align*} 
	By the definition of regime $\mathcal{X}_{\text{II}},$ $X+\lambda >R_3(X)>0$ and since $g$ is symmetric and strictly decreasing on $(0, \infty),$ $g(X+\lambda)< g(R_3(X))=g(-R_3(X))=  g(X-\operatorname{U}_{\alpha}(X)).$ This completes the proof for (ii).

	(iii) Suppose $X \in \mathcal{X}_{\text{III}}.$ For this regime to be non-empty, we must have $R_3(X)\geq \lambda.$ In this case $\operatorname{U}_{\alpha}(X)=X+R_3(X)$ and $\operatorname{L}_{\alpha}(X) =-\operatorname{U}_{\alpha}(-X) =X-R_3(X),$ thanks to the symmetry $R_3(X)=R_3(-X).$ By definition of regime $\mathcal{X}_{\text{III}},$  $\operatorname{U}_{\alpha}(X) \geq \lambda$ and $\operatorname{L}_{\alpha}(X) \leq -\lambda.$ For the posterior mass, we find with $G(-q)=1-G(q)$ and $R_3(X) =G^{-1}(1- \tfrac \alpha 2 D(X)),$
	\begin{align*}
		\Pi\big([\operatorname{L}_{\alpha}(X),\operatorname{U}_{\alpha}(X)] \setminus \{0\} \big | X \big) 
		&= 
		1- \Pi(0|X)- \Pi\big((-\infty,\operatorname{L}_{\alpha}(X))  \big | X \big)  - \Pi\big((\operatorname{U}_{\alpha}(X), \infty) \big | X \big) \\
		&= 1 - \Pi(0|X) - \frac{G(\operatorname{L}_{\alpha}(X)-X)}{D(X)} - \frac{1 - G(\operatorname{U}_{\alpha}(X)-X)}{D(X)} \\
		&= 1 - \Pi(0|X) + \frac{2G(R_3(X))-2}{D(X)} \\
		&=1 -\Pi(0|X) -\alpha.
	\end{align*}
	To see that $[\operatorname{L}_{\alpha}(X),\operatorname{U}_{\alpha}(X)]$ is a level set, we can argue as in the proof for $(i).$
	
	(iv) From \eqref{eq.posterior} and the symmetry $g(x)=g(-x)$ for any $x,$ it follows that if $A$ is a posterior level set given observation $X=x,$ then $-A=\{-a: a\in A\}$ is a level set with the same posterior probability given that we observe $X=-x.$ If $x\in \mathcal{X}_{\text{IV}},$ we have $-x \in \mathcal{X}_{\text{II}}$ and $\operatorname{L}_{\alpha}(x)=-\operatorname{U}_{\alpha}(-x)=x-R_2(-x).$ With $(ii),$
	\begin{align*}
		\Pi\big([\operatorname{L}_{\alpha}(x), \operatorname{U}_{\alpha}(x)] \, \big | \, X=x\big) 
		&= \Pi\big([x-R_2(-x), -\lambda] \, \big | \, X=x\big) \\
		&= \Pi\big([\lambda, -x+R_2(-x)] \, \big | \, X=-x\big ) \\
		&= 1-\Pi(0 |X=-x) -\alpha.
	\end{align*}
	Moreover, \eqref{eq.post_point_mass} shows that $\Pi(0 |X=-x)=\Pi(0 |X=x)$ and this completes the proof for $(iv).$
\end{proof}

\begin{proof}[Proof of Lemma \ref{lem.basic_props_Ualpha}]
	{\it (i):} From \eqref{eq.Rsdef}, we obtain $R_1(x) \leq G^{-1}(1-\alpha/2)$ which together with the definition of regime $\mathcal{X}_{\textnormal{I}}$ yields the conclusion. 
	
	{\it (ii):} We need to show that $R_3(0)\geq \lambda.$ By rewriting we find that $\Pi(0|X=0) < 1-\alpha$ implies $\alpha (1-w)g(0)/(2w) < (1-\alpha)G(-\lambda).$ With \eqref{eq.Rsdef}, 
	\begin{align*}
		R_3(0) > G^{-1}\big( 1 -(1-\alpha) G(-\lambda) -\alpha G(-\lambda) \big) = \lambda.
	\end{align*}
	Hence $0 \in \mathcal{X}_{\text{III}}.$
	
	{\it (iii):} By $(i)$ and $(ii)$ it remains to show that $\mathcal{X}_{\text{II}}$ (and thereby $\mathcal{X}_{\text{IV}}$) is non-empty. In a first step, we show that that there exists a solution $\lambda + R_1(x^*)-x^*=0.$ Since $g\in \mathcal{G},$ $R_1$ is continuous and we will apply the intermediate value theorem to the function $\lambda + R_1(x)-x.$ For $x=0,$ we have that the value of the function is positive since $t_\alpha = -\infty$ and $R_1(x)\geq 0$ by Lemma \ref{lem.well_def}. For $x \uparrow \infty$ we use $\lambda + R_1(x)-x\leq \lambda + G^{-1}(1-\alpha/2)-x$ to see that the function eventually becomes negative. Thus, $\lambda + R_1(x^*)-x^*=0$ for some $x^*\geq 0$ by the intermediate value theorem. 
	
	Now we prove that for this solution $x^*,$ $-\lambda + R_3(x^*) < x^*,$ therefore implying $x^*\in \mathcal{X}_{\text{II}}$ by definition of the regime $\mathcal{X}_{\text{II}}.$ Using \eqref{eq.Rsdef}, we have for any $x,$ $G(R_3(x))=G(R_1(x))+\Delta_{\lambda}(x)/2.$ Since $R_1(x^*)=x^*-\lambda,$ rewriting $\Delta_{\lambda}(x^*)$ and using the monotonicity of $G$ yields
	\begin{align*}
		G(R_3(x^*)) = G(x^* - \lambda) +\frac{1}{2}\Delta_{\lambda}(x^*) = \frac{G(x^* - \lambda) + G(x^* + \lambda)}2 < G(x^* + \lambda) \;.
	\end{align*}
	Thus also $-\lambda + R_3(x^*) < x^*$ and hence $x^* \in \mathcal{X}_{\text{II}}$. 
	
	{\it (iv):} Using (i) and the definition of $\mathcal{X}_{\textnormal{III}},$ it is enough to show that $R_3(x)<\lambda$ for all $|x|\leq \lambda+G^{-1}(1-\alpha/(1+\alpha))=:\widetilde{x}.$ Observe that for any $|x|\leq \widetilde{x},$ $R_3(x)\leq G^{-1}(1-\tfrac{1-w}{2w}\alpha g(\widetilde{x})).$ For all sufficiently small $w,$ the right hand side is strictly smaller than $\lambda.$
	
	{\it (v):} In part (a) of the proof we show that $\mathcal{X}_{\textnormal{III}}$ is the empty set and in part (b) we show that $t_\alpha < \lambda$ for all sufficiently large $\lambda.$ Part (c) combines the results from (a) and (b).
	
	(a) By symmetry, it is sufficient to consider $x > 0$. We show that for sufficiently large $t_\alpha$ and all $x>t_\alpha \vee 0,$ we have $R_3(x)< x+\lambda$ which then implies that $\mathcal{X}_{\textnormal{III}}$ must be empty. By Lemma \ref{lem.tech.delta} (iii), $x\mapsto 1 - \Delta_{\lambda}(x) = G(x - \lambda) + G(-x - \lambda)$ and $x\mapsto 1/g(x)$ are strictly increasing on $(0, \infty).$ Recall that $t_\alpha=t_\alpha(g, \lambda,w)$ is the solution to  
	\begin{align}
		\frac{w}{1-w} \frac{G(t_{\alpha}-\lambda)+G(-t_{\alpha}-\lambda)}{g(t_{\alpha})} = \frac{\alpha}{1-\alpha}.
		\label{eq.eqeq}
	\end{align}
	Thus, by increasing $\lambda,$ we can make $t_\alpha$ arbitrary large. In particular, we choose $\lambda^*$ such that for any $\lambda\geq \lambda^*,$ $t_\alpha$ is such that $G(-\lambda-t_\alpha)- \tfrac{\alpha}2 G(t_\alpha-\lambda)<0.$ Using the monotonicity of $G,$ this also implies that $G(-\lambda-x)- \tfrac{\alpha}2 G(x-\lambda)<0$ for all $x\geq t_\alpha\vee 0.$ This yields the second inequality in
	\begin{align*}
		R_3(x)\leq G^{-1}\Big(1-\frac{\alpha}{2}+\frac{\alpha}{2}G(\lambda-x)\Big) < x+\lambda,
	\end{align*}
	using \eqref{eq.Rsdef} for the first inequality together with $G(x + \lambda) = 1 - G(-\lambda - x)$ and $G(\lambda - x) = 1 - G(x - \lambda).$ This completes the proof for (a).
	
	(b) We show that $t_\alpha < \lambda$ for large $\lambda$ by contradiction. Thus, suppose $t_{\alpha} \geq \lambda$. By the monotonicity properties of $1/g$ and $G$ used in (a),
	\begin{align*}
		\frac{w}{1-w} \frac{G(t_{\alpha} - \lambda) + G(-t_{\alpha} - \lambda)}{g(t_{\alpha})} \geq \frac{w}{1-w} \frac{G(\lambda - \lambda) + G(-\lambda - \lambda)}{g(\lambda)}.
	\end{align*}
	For all sufficiently large $\lambda,$ the right hand side is strictly larger than $\alpha/(1-\alpha)$. This is a contradiction to the fact that $t_\alpha$ is a solution of \eqref{eq.eqeq}. Hence $t_\alpha < \lambda.$
	
	(c) Since $\mathcal{X}_{\textnormal{I}}$ only contains $x$ with $|x|>t_\alpha \vee \lambda,$ it follows that $(t_\alpha, \lambda] \cap  \mathcal{X}_{\textnormal{I}}$ is empty. Since $\mathcal{X}_{\textnormal{IV}} \cap \{x:x>t_\alpha\vee 0\}=\varnothing,$ we conclude that $(t_{\alpha}, \lambda] \subseteq \mathcal{X}_{\textnormal{II}}$.
\end{proof}

\begin{proof}[Proof of Theorem \ref{thm.altern_Ua}]
	We show that the formula holds for each of the regimes with regime $\mathcal{X}_{\textnormal{I}}$ being further subdivided into positive and negative $x.$
	
	Suppose that $x> \lambda +R_1(x)$ and $x>t_\alpha.$ By the definition of $R_1, R_3$ in \eqref{eq.Rsdef}, we have $R_1(x) <R_3(x)$ for all $x.$ By  Lemma \ref{lem.sign.equiv} (i), we can conclude that for all $x> \lambda +R_1(x),$ $R_1(x) = (R_3(x) \wedge R_2(x)) \vee R_1(x).$ By Theorem \ref{thm.LU}, $\operatorname{U}_{\alpha}(x) = H_1(x)$ for $x> \lambda +R_1(x).$ Since by Lemma \ref{lem.well_def}, $R_1(x)\geq 0$ we have moreover in this regime $H_1(x)=x+R_1(x)>\lambda.$
	
	Consider now $x\in \mathcal{X}_{\textnormal{II}}.$ By definition of the regime, $-\lambda+R_3(x) < x \leq \lambda +R_1(x).$ By Lemma \ref{lem.sign.equiv}, we have thus $R_1(x) \leq R_2(x) < R_3(x)$ and $\operatorname{U}_{\alpha}(x) = x+R_2(x) = x+ (R_2(x) \wedge R_3(x)) \vee R_1(x) =H_1(x).$ By \eqref{eq.R2_satisfies}, $H_1(x)=x+R_2(x) > \lambda.$
	
	Next we study $x \in \mathcal{X}_{\textnormal{III}}.$ We show that then $R_3(x) \leq R_2(x).$ Notice that $R_2(-u)\geq R_2(u)$ for all $u\geq 0.$ Since also $R_3(x)=R_3(-x),$ it is enough to show the inequality for $x\geq 0.$ This, however, follows immediately from Lemma \ref{lem.sign.equiv} (ii). Since it always holds that $R_1(x)< R_3(x),$ we obtain $\operatorname{U}_{\alpha}(x) = x+R_3(x) =H_1(x).$ Let us now prove that also $x+R_3(x)>\lambda$ in this case. Suppose not. The credible set is $[\operatorname{L}_{\alpha}(x),\operatorname{U}_{\alpha}(x)]=[x-R_3(x),x+R_3(x)]$ in this regime. If $x+R_3(x) \leq \lambda,$ the posterior coverage of $[\operatorname{L}_{\alpha}(x),\operatorname{U}_{\alpha}(x)]\setminus \{0\}$ is zero. It has to be, however, that $1-\alpha -\Pi(0|X=x)>0.$ This is a contradiction and we must have $x+R_3(x)=H_1(x)>\lambda.$
	
	If $x \in \mathcal{X}_{\textnormal{IV}},$ then $\operatorname{U}_{\alpha}(x)=-\lambda$ and $-x \in \mathcal{X}_{\textnormal{II}}.$ The latter implies that $-x\leq \lambda+R_1(-x)=\lambda+R_1(x).$ Thus $H_2(x)=(-\lambda)\wedge (x+R_1(x))=-\lambda=\operatorname{U}_{\alpha}(x).$ For this regime, it remains to show that $H_1(x)<\lambda.$ Since $R_1(x) \leq R_3(x),$ for all $x,$ we also have that $H_1(x) \leq x+R_3(x).$ Since $-x \in \mathcal{X}_{\textnormal{II}},$ we also have $-\lambda+R_3(x)< -x$ which combined with the previous inequality gives $H_1(x) < \lambda.$
	
	Finally suppose that $-x > \lambda+R_1(x)$ and $-x>t_\alpha.$ Thus $x \in \mathcal{X}_{\textnormal{I}},$ $\operatorname{U}_{\alpha}(x) = x+R_1(x)< -\lambda$ and $\operatorname{U}_{\alpha}(x)=H_2(x).$ To show that $H_1(x)< \lambda,$ observe that by Lemma \ref{lem.partition}, $x\notin \mathcal{X}_{\textnormal{III}}.$ By definition of $\mathcal{X}_{\textnormal{III}},$ $x\notin \mathcal{X}_{\textnormal{III}}$ and $x<0$ implies $-x>-\lambda+R_3(x).$ Arguing as above, we thus have $H_1(x) \leq x+R_3(x) < \lambda.$
	
	As we have treated all possible cases, the proof is complete.
\end{proof}

\begin{proof}[Proof of Lemma \ref{lem.Ua_decrease}]
	We use the following closed forms for the Laplace distribution: $g(x) = \frac{1}{2}\exp(-|x|)$ for any real $x$, $G(x) = \frac{1}{2}\exp(x) = g(-x)$ for $x \leq 0$ and $G^{-1}(p) = -\ln(2(1-p))$ for $p \in [\frac{1}{2},1)$.
	
	The interval $( \tfrac 12 \ln(\tfrac{2}{\alpha}) ,  \ln(\frac{1-\alpha}{\alpha} \frac{w}{1-w}))$ is non-empty. To see this observe that $\sqrt{\frac{2}{\alpha}} < \frac{1-\alpha}{\alpha} \frac{w}{1-w}$ for $\alpha <\frac{1}{2}$ and $w \in (\sqrt{2\alpha},1].$
	
	Next, we show that $t_{\alpha} =-\infty,$ that is, $\Pi(0|X=0)< 1-\alpha.$ From the formula for the posterior and using that $1-\Delta_{\lambda}(0)= 2G(-\lambda),$ it is sufficient to verify that
	\begin{align} \label{eq.laplace.ta}
		\frac{2w G(- \lambda)}{(1-w)g(0)} > \frac{\alpha}{1-\alpha} \;.
	\end{align}
	Since $g$ and $G$ are the p.d.f. and c.d.f. of the Laplace distribution, the previous inequality is equivalent to $\frac{2w}{1-w} \exp(-\lambda) > \frac{\alpha}{1-\alpha}$. This clearly holds for $\lambda <  \ln(\frac{1-\alpha}{\alpha} \frac{w}{1-w})$ proving that $\Pi(0|X=0)< 1-\alpha$ and thus $t_\alpha=-\infty.$

	We prove now that $(\tfrac 12\ln(\frac{2}{\alpha}),\lambda) \subset \mathcal{X}_{\textnormal{II}}$. By definition of $R_3$ and $\Delta_\lambda,$
	\begin{align*}
		R_3(x) < G^{-1}\Big(1 - \frac{\alpha}{2}\big(1 - G(\lambda - x)\big)\Big)= G^{-1}\Big(1 - \frac{\alpha}{2}G(x-\lambda)\Big) \; .
	\end{align*}
	For $x< \lambda,$ we have $1 - \frac{\alpha}{2} G(x-\lambda)> 1/2$ and with $G^{-1}(p) = -\ln(2(1-p))$ for $p \in [\frac{1}{2},1),$ the right hand side of the previous display becomes
	\begin{align*}
		- \log\Big(\frac{\alpha}{2}\exp(x-\lambda)\Big) = \ln\Big(\frac{2}{\alpha}\Big) - x + \lambda \;.
	\end{align*}
	Hence $R_3(x) < x + \lambda$ if $x \in (\tfrac 12 \ln(\frac{2}{\alpha}), \lambda).$ This interval is thus not in $\mathcal{X}_{\textnormal{III}}$. Because $0<x < \lambda$ also implies $x \not \in \mathcal{X}_{\textnormal{I}},$ we conclude by Lemma \ref{lem.partition} that $(\tfrac 12\ln(\frac{2}{\alpha}),\lambda) \subset \mathcal{X}_{\textnormal{II}}.$
	
	As a final step, we now show that $\operatorname{U}_{\alpha}$ is decreasing on $(\tfrac 12 \ln(\frac{2}{\alpha}), \lambda) \subset \mathcal{X}_{\textnormal{II}}$. Let $x\in (\tfrac 12 \ln(\frac{2}{\alpha}), \lambda).$ By Theorem \ref{thm.LU}, $\operatorname{U}_{\alpha}(x) = x + R_2(x)$ and
	\begin{align*}
		\operatorname{U}_{\alpha}'(x) = 1 - \frac{\alpha \frac{1-w}{w}g'(x) + \alpha g(\lambda - x) + (1-\alpha)g(-\lambda - x)}{g(R_2(x))}=:1-B \;.
	\end{align*}
	For the Laplace$(0,1)$ density, we have $g(x) = -g'(x) < 0$ for any $x > 0$. Hence, the numerator of the fraction $B$ can be rewritten as
	\begin{align*}
		- \alpha\frac{1-w}{2w}e^{-x} + \frac{\alpha}{2}e^{x- \lambda} + \frac{1-\alpha}{2}e^{-\lambda - x} =: -s_1 + s_2 + s_3\;,
	\end{align*}
	For the denominator of the fraction, use that $R_2(x) \geq 0$ for any $x \in \mathcal{X}_{\textnormal{II}}$ and $g(G^{-1}(p)) = 1-p$ for $p \in [\frac{1}{2},1)$ for the Laplace distribution. Applying the definition of $R_2$, we have $g(R_2(x)) =  \alpha \frac{1-w}{w}g(x) + \alpha G(x - \lambda)  - (1-\alpha)G(-\lambda - x)$, using that $1 - G(\lambda - x) = G(x - \lambda).$ For the Laplace distribution, we thus find
	\begin{align*}
		g(R_2(x)=\alpha \frac{1-w}{2w}e^{-x} + \frac{\alpha}{2}e^{x - \lambda} - \frac{1-\alpha}{2}e^{-\lambda - x} = s_1 + s_2 - s_3\;.
	\end{align*}
	Since by assumption $\lambda < \ln(\frac{1-\alpha}{\alpha} \frac{w}{1-w}),$ we have that
	\begin{align*}
		s_1 = \alpha\frac{1-w}{2w}e^{-x} <  \frac{1-\alpha}{2}e^{-x-\lambda} = s_3\; ,
	\end{align*}
	it follows that $\operatorname{U}_{\alpha}'(x) = 1- (-s_1 + s_2 + s_3)/(s_1 + s_2 - s_3)<0.$ This completes the proof. 
\end{proof} 

\subsection{Proofs for Section \ref{sec.coverage}}

\begin{proof}[Proof of Lemma \ref{lem.decrease_cov}]
	Recall that for most notation, we omitted the dependence on  the mixing weight on the slab prior distribution $w.$ By definition, we have that the functions $R_1,R_2,R_3$ are monotone increasing in $w$ on any point $x$ these functions are defined on. Thus, the functions $H_1$ and $H_2$ in Theorem \ref{thm.altern_Ua} are decreasing in $w.$ Since $H_2 \leq 0$ and $H_1\geq 0,$ also $\operatorname{U}_{\alpha}$ is decreasing in $w$ on any point $|x|>t_\alpha.$  Because of $\operatorname{L}_{\alpha}(x)=-\operatorname{U}_{\alpha}(-x),$  $\operatorname{L}_{\alpha}$ must be increasing in $w$ on any point $|x|>t_\alpha.$ Using that $g\in \mathcal{G},$ \eqref{eq.pt_mass_zero} implies that $t_\alpha$ increases if $w$ decreases. Thus, for any $x,$ the set $[\operatorname{L}_{\alpha}(x),\operatorname{U}_{\alpha}(x)] \cap \{|x| > t_{\alpha}\}$ becomes smaller if $w$ decreases. This decreases the probability on the right hand side of \eqref{eq.credible_rewritten} which coincides with the frequentist coverage.
\end{proof}

\begin{proof}[Proof of Lemma \ref{lem.sum_cov}]
	We show that for any $|x|>t_\alpha,$ $\theta_0 \in [\operatorname{L}_{\alpha}(x),\operatorname{U}_{\alpha}(x)]$ if and only if either $\theta_0\in [\operatorname{L}_{\alpha}(x),x]$ or $\theta_0 \in (x,\operatorname{U}_{\alpha}(x)].$ If $x\in \mathcal{X}_{\textnormal{I}} \cup \mathcal{X}_{\textnormal{III}},$ then by Lemma \ref{lem.well_def} and Theorem \ref{thm.LU}, $\operatorname{L}_{\alpha}(x) < x < \operatorname{U}_{\alpha}(x)$ and the claim holds. If $x\in \mathcal{X}_{\textnormal{II}},$ then, $\operatorname{L}_{\alpha}(x)=\lambda$ and by arguing as in the proof of Theorem \ref{thm.altern_Ua}, $R_2(x)\geq R_1(x)>0.$ There are two cases. First if $\operatorname{L}_{\alpha}(x)\leq x< x+R_2(x)=\operatorname{U}_{\alpha}(x)$ the claim holds immediately. The second case is that $x<\operatorname{L}_{\alpha}(x)<\operatorname{U}_{\alpha}(x).$ Then, $[\operatorname{L}_{\alpha}(x),x]$ is empty. Moreover, because of $\theta_0>\lambda,$ we have $\theta_0 \in [\operatorname{L}_{\alpha}(x),\operatorname{U}_{\alpha}(x)]=[\lambda,\operatorname{U}_{\alpha}(x)]$ if and only if $\theta_0 \in (x,\operatorname{U}_{\alpha}(x)].$ Thus the claim also follows for $\mathcal{X}_{\textnormal{II}}.$ Finally in regime $\mathcal{X}_{\textnormal{IV}},$ we have $\operatorname{L}_{\alpha}(x)<\operatorname{U}_{\alpha}(x)=-\lambda.$ For the claim to hold it is enough to check that $x\leq \lambda.$ This holds since $-x\in \mathcal{X}_{\textnormal{II}}$ and thus $x\leq \lambda-R_3(x)<\lambda.$
\end{proof}

\begin{proof}[Proof of Lemma \ref{lem.cov_reform}]
	In a first step, we show that on $\{x:|x|>t_\alpha\},$ the functions $\operatorname{U}_{\alpha}$ and $\operatorname{L}_{\alpha}$ can only jump from $-\lambda$ to $\lambda$ or back and are otherwise continuous. Because of $\operatorname{L}_{\alpha}(x)=-\operatorname{U}_{\alpha}(-x),$ it is enough to verify this for $\operatorname{U}_{\alpha}.$ By Theorem \ref{thm.altern_Ua} and \eqref{eq.Rsdef}, $H_1$ and $H_2$ are continuous and a jump in $\operatorname{U}_{\alpha}$ can only occur from a value $\leq -\lambda$ to $\lambda$ (or the other way around). From the proof of Theorem \ref{thm.altern_Ua}, we know that $\operatorname{U}_{\alpha}(x)<\lambda$ if and only if $x\in (\mathcal{X}_{\textnormal{I}} \cap (-\infty,0])\cup \mathcal{X}_{\textnormal{IV}}=:\mathcal{Y}.$ On $\mathcal{Y},$ $\operatorname{L}_{\alpha}$ is continuous. Thus, a jump from a value $<-\lambda$ to $\lambda$ would lead to a strict increase of the posterior credibility, which contradicts the fact that $\{0\}\cup ([\operatorname{L}_{\alpha}(x),\operatorname{U}_{\alpha}(x)]\setminus(-\lambda,\lambda))$ are $(1-\alpha)$-credible sets for all $x.$ Thus, if there is a jump it has to be from $-\lambda$ to $\lambda.$ 
	
	Next we study what happens with $\operatorname{U}_{\alpha}$ and $\operatorname{L}_{\alpha}$ in $[-t_\alpha,t_\alpha].$ We must have $\lim_{x\downarrow t_\alpha} \operatorname{U}_{\alpha}(x)=\lim_{x\downarrow t_\alpha} \operatorname{L}_{\alpha}(x)$ as otherwise the $\operatorname{HPD}_\alpha(X)$ credible set would cover more than $1-\alpha$ of the posterior mass for $X\downarrow t_\alpha.$ By checking all regimes individually, the only values that can occur for the limits are $\{t_\alpha, \lambda\}.$ For the same reason also $\lim_{x\uparrow -t_\alpha} \operatorname{U}_{\alpha}(x)=\lim_{x\uparrow -t_\alpha} \operatorname{L}_{\alpha}(x)\in \{-t_\alpha, -\lambda\}.$ Thus, going from negative $x$ with $x<-t_\alpha$ to positive $x$ with $x>t_\alpha$ induces a jump in the functions $x\mapsto \operatorname{L}_{\alpha}(x)$ and $x \mapsto \operatorname{U}_{\alpha}(x)$ from a function value $\{-t_\alpha, -\lambda\}$ to $t_\alpha$ or $\lambda.$ Except for these discontinuities that only affect the function values in $[-(t_\alpha \vee \lambda) , t_\alpha \vee \lambda],$ the functions $\operatorname{L}_{\alpha}$ and $\operatorname{U}_{\alpha}$ are otherwise continuous by Theorem \ref{thm.altern_Ua}. Since $\lim_{x\downarrow -\infty}\operatorname{L}_{\alpha}(x)=\lim_{x\downarrow -\infty}\operatorname{U}_{\alpha}(x)=-\infty$ and $\lim_{x\uparrow \infty}\operatorname{L}_{\alpha}(x)=\lim_{x\uparrow \infty}\operatorname{U}_{\alpha}(x)=\infty$ and $\theta_0>t_\alpha,$ the intermediate value theorem shows that if there are $x,y$ with $|x|,|y|>t_\alpha$ and $\operatorname{L}_{\alpha}(x)< \theta_0  <\operatorname{L}_{\alpha}(y),$ then there exists $z$ between $x$ and $y,$ such that $|z|>t_\alpha$ and $\operatorname{L}_{\alpha}(z)=\theta_0.$ The same holds also with $\operatorname{L}_{\alpha}$ replaced by $\operatorname{U}_{\alpha}.$
	
	This shows that for $\theta_0>\lambda \vee t_\alpha$ the sets $\operatorname{L}_{\alpha}^{-1}(\theta_0)$ and $\operatorname{U}_{\alpha}^{-1}(\theta_0)$ are non-empty.
	
	In a next step we show that $\theta_0>t_\alpha$ and $\theta_0 \in [\operatorname{L}_{\alpha}(X),X]$ imply that $\theta_0\leq X\leq \sup \operatorname{L}_{\alpha}^{-1}(\theta_0),$ thus proving the upper bound for $C^{-}(\theta_0).$ It suffices to show that $X\leq \sup \operatorname{L}_{\alpha}^{-1}(\theta_0).$ Suppose this is not true and there exists $x^*$ with $|x^*|>t_\alpha,$ satisfying $\operatorname{L}_{\alpha}(x^*) \leq \theta_0$ and $x^*> \sup \operatorname{L}_{\alpha}^{-1}(\theta_0).$ If $\operatorname{L}_{\alpha}(x^*)=\theta_0$ we have a contradiction, thus we can even assume that $\operatorname{L}_{\alpha}(x^*)< \theta_0.$ By the version of the intermediate value theorem proved above, there exists $z>x^*$ with $|z|>t_\alpha$ and $\operatorname{L}_{\alpha}(z)=\theta_0,$ again contradicting $x^*> \sup \operatorname{L}_{\alpha}^{-1}(\theta_0).$ This establishes the upper bound on $C^{-}(\theta_0).$
	
	By following the same arguments as above and using that by assumption $\theta_0>t_\alpha$, one can also prove that $\theta_0\leq X\leq \inf \operatorname{L}_{\alpha}^{-1}(\theta_0)$ implies $\theta_0\in [\operatorname{L}_{\alpha}(X),X].$ Therefore also $\theta_0\in [\operatorname{L}_{\alpha}(X),X]\cap \{|X|>t_\alpha\}.$ This proves the lower bound on $C^{-}(\theta_0).$
	
	The upper and lower bound on $C^{+}(\theta_0)$ can be shown following the same reasoning.
\end{proof}

\begin{proof}[Proof of Lemma \ref{lem.compute}]
	By Lemma \ref{lem.well_def} and Lemma \ref{lem.mon}, the function $R_1$ is positive and monotone increasing on $|x|>t_\alpha$. We show by induction that $a_{k+1}>a_k.$ This is true for $k=0.$ Suppose $a_k>a_{k-1},$ then it follows from the monotonicity of $R_1$ that $a_{k+1}>a_k,$ completing the induction argument. In a second step, we show using induction again, that $a_k\leq  \inf \operatorname{L}_{\alpha}^{-1}(\theta_0)$ for all $k.$ Due to $R_1 \geq 0,$ $a_0\leq  \inf \operatorname{L}_{\alpha}^{-1}(\theta_0).$ For the inductive step, suppose that $a_k\leq  \inf \operatorname{L}_{\alpha}^{-1}(\theta_0)$ for a given $k.$  Then, $R_1(a_k) \leq R_1(\inf \operatorname{L}_{\alpha}^{-1}(\theta_0)),$ hence $\inf \operatorname{L}_{\alpha}^{-1}(\theta_0)-R_1(a_k)\geq \theta_0$ and consequently $a_{k+1}\leq \inf \operatorname{L}_{\alpha}^{-1}(\theta_0)$ completing the inductive step. Since $(a_k)_k$ is increasing and bounded it must have a limit $a=\lim_k a_k$ and this limit satisfies $a=\theta_0+R_1(a)$ implying $a \in \operatorname{L}_{\alpha}^{-1}(\theta_0).$ Since $a_k\leq  \inf \operatorname{L}_{\alpha}^{-1}(\theta_0)$ for all $k,$ we conclude that $a=\inf \operatorname{L}_{\alpha}^{-1}(\theta_0).$
\end{proof}

\begin{proof}[Proof of Lemma \ref{lem.exam}]
	Clearly $g \in \mathcal{G}.$ For all $u \geq  0$ it follows that $G(u) \geq 1/2$, hence $G(3u/2) \leq 1^{1+\gamma} \leq \big(2 G(u) \big)^{1+\gamma}$ for all $\gamma > 0.$ If also $G(3u/2)\leq cG(u)^{1+\gamma}$ for all $u\leq 0,$ then $G(3u/2)\leq (2^{1+\gamma}\vee c) G(u)^{1+\gamma}$ for all real $u.$
	
	Using substitution, $x\leq u \leq 0,$ and $1+\gamma <(3/2)^\eta,$
	\begin{align*}
		G\Big(\frac 32 u\Big) 
		&=\int_{-\infty}^{3u/2} c_\eta e^{-|x|^\eta} \, dx
		= \frac{2}{3} \int_{-\infty}^{u} c_\eta e^{-(\frac 32)^\eta |x|^\eta} \, dx \\
		&\leq \frac{2}{3}c_\eta e^{-(1+\gamma) |u|^\eta} \int_{-\infty}^u e^{(1+\gamma - (\frac 32)^\eta)|x|^\eta} \, dx
		\leq \frac{2}{3}c_\eta C(\gamma,\eta) e^{-(1+\gamma) |u|^\eta},
	\end{align*}
	with $C(\gamma,\eta):= \int_{-\infty}^0 e^{(1+\gamma - (\frac 32)^\eta)|x|^\eta} \, dx< \infty.$ On the other hand, using that $|u|^\eta-|x|^\eta \geq -|u-x|^\eta$ for any $0\leq \eta \leq 1,$ and the fact that $\int_{-\infty}^0 e^{-|x|^\eta} \, dx =1/(2c_\eta),$ we have
	\begin{align}
		G(u)= \int_{-\infty}^u c_\eta e^{-|x|^\eta} \, dx
		\geq  c_\eta e^{-|u|^\eta} \int_{-\infty}^u  e^{-|u-x|^\eta} \, dx = \frac 12 e^{-|u|^\eta}.
		\label{eq.exam2}
	\end{align}
	Combining the last two displays, we have that $G(3u/2)\leq cG(u)^{1+\gamma}$ for all $u\leq 0$ and $c$ a sufficiently large constant.
	
	The second part of the condition holds for any $\gamma \leq 1$ using $1-G(q)=G(-q)$ and \eqref{eq.exam2}.
\end{proof}

\begin{lem}[Upper bound on $\Delta_\lambda(x)$ in regime $\mathcal{X}_{\textnormal{I}}$] \label{lem.tech.general.lambda-xl}
	For $g\in \mathcal{G}$ and any $x \in \mathcal{X}_{\textnormal{I}}$ with $|x| > t_{\alpha}$,
	\begin{align*}
		\Delta_\lambda(x) < G(\lambda - x) < \frac{\alpha}{1+\alpha}\Big(1+\frac{1-w}{w} g(x)\Big) \;.
	\end{align*}
\end{lem}	
\begin{proof}
	The first inequality follows immediately from the definition of $\Delta_\lambda(x).$ By the definition of regime $\mathcal{X}_{\textnormal{I}},$ we have $\lambda - x < - R_1(x)$ for $x \in \mathcal{X}_{\textnormal{I}}.$ Using $-G^{-1}(1-p)=G^{-1}(p)$ together with the definition of $R_1,$ we obtain
	\begin{align*}
		G(\lambda - x) < G(-R_1(x)) = \frac{\alpha}{2}\Big(1 + \frac{1-w}{w}g(x)\Big) + \frac{1-\alpha}{2}\Delta_{\lambda}(x) \;.
	\end{align*} 
	With $\Delta_{\lambda}(x) < G(\lambda - x),$ the second inequality follows by rewriting.
\end{proof}

\begin{lem}\label{lem.tech.g(R1)}
	Suppose Assumption \ref{def.tail.decay} holds, for some $c_*$ and $\gamma \in (0,1],$ then, we have for any $x \in \mathcal{X}_{\textnormal{I}},$
	\begin{align*}
		g\big(R_1(x)\big) \leq c_* \alpha^{\gamma} \Big(1 + \frac{1-w}{w}g(0)  \Big)
	\end{align*}
\end{lem}

\begin{proof}
	By Assumption \ref{def.tail.decay} and the definition of the function $R_1,$
	\begin{align*}
		g\big(R_1(x)\big) 
		&\leq c_* \big(1 -G(R_1(x))\big)^\gamma
		= c_*\Big(\frac{\alpha}{2}+\frac{1-w}{2w}\alpha g(x) +\frac{1-\alpha}{2}\Delta_\lambda(x)\Big)^\gamma.
	\end{align*}
	Since $x\in \mathcal{X}_{\textnormal{I}},$ we find using Lemma \ref{lem.tech.general.lambda-xl}, $\Delta_\lambda(x) < \tfrac{\alpha}{1+\alpha}(1+\tfrac{1-w}{w} g(x)).$ Rearranging the terms yields,
	\begin{align*}
		g\big(R_1(x)\big) 
		&\leq c_*\Big(\frac{\alpha}{\alpha+1}\Big( 1 + \frac{1-w}{w}g(x)   \Big)  \Big)^\gamma
		\leq c_* \alpha^\gamma \Big( 1 + \frac{1-w}{w}g(0)\Big)^\gamma.
	\end{align*}
\end{proof}

\begin{lem}[Lower bound on $R_1$] \label{lem.tech.R1}
	For $g \in \mathcal{G},$
	\begin{align*}
		\inf_{x\in \mathcal{X}_{\textnormal{I}}, |x|>t_\alpha} R_1(x) \geq G^{-1}\Big(1 - \frac{\alpha}{1+\alpha}\Big(1 + \frac{1-w}{w}g(\lambda)\Big)\Big).
	\end{align*}	
\end{lem}
\begin{proof}
	Using the definition of $R_1$ in \eqref{eq.Rsdef}, $\Delta_{\lambda}(x) = G(\lambda - x) - G(-\lambda - x)\leq G(\lambda - x),$ and Lemma \ref{lem.tech.general.lambda-xl}, we obtain
	\begin{align*}
		R_1(x) 
		&= 
		G^{-1}\Big( 1- \frac{\alpha}2 - \frac{1-w}{2w} \alpha g(x) - \frac{1-\alpha}2 \Delta_\lambda(x) \Big)	 \\
		&\geq G^{-1}\Big(1 - \frac{\alpha}{2}\Big(1 + \frac{1-\alpha}{1+\alpha}\Big)\Big(1 + \frac{1-w}{w}g(x)\Big) \Big) \;.
	\end{align*}
	Because of $x\in \mathcal{X}_{\textnormal{I}},$ $x>\lambda+R_1(x)\geq \lambda,$ completing the proof.
\end{proof}

\begin{proof}[Proof of Proposition \ref{prop.upper_coverage}]
	Recall that $\operatorname{L}_{\alpha}^{-1}(\theta_0) \subseteq \mathcal{X}_{\textnormal{I}}.$ Throughout the proof we frequently use the bounds in Lemma \ref{lem.cov_reform} and argue as for \eqref{eq.coverage_bds_specific}.
	
	(i) By \eqref{eq.Rsdef}, we have that $R_1(x) \leq G^{-1}(1-\alpha/2)$ and thus $\sup \operatorname{L}_{\alpha}^{-1}(\theta_0) \leq \theta_0+G^{-1}(1-\alpha/2).$ Using Lemma \ref{lem.cov_reform}, we conclude that 
	\begin{align*}
		C^{-}(\theta_0) 
		\leq \mathbb{P}_{\theta_0}\big(\theta_0\leq X\leq \theta_0+G^{-1}(1-\alpha/2)\big)
		= G\big( G^{-1}(1-\alpha/2)\big) - G(0) = \frac{1-\alpha}{2}
	\end{align*}
	establishing the upper bound. For the lower bound, we have with $\widetilde x :=\inf \operatorname{L}_{\alpha}^{-1}(\theta_0),$
	\begin{align}\label{eq.sdf1}
		\begin{split}
			C^{-}(\theta_0) 
			&\geq 
			\mathbb{P}_{\theta_0}\big( \theta_0\leq X\leq \theta_0+R_1(\widetilde x) \big) \\
			&= G\big(R_1(\widetilde x) \big) - \frac 12  \\
			&= \frac{1-\alpha}{2} - \frac{1-w}{2w}\alpha g(\widetilde x) - \frac{1-\alpha}{2}\Delta_\lambda(\widetilde x).
		\end{split}	
	\end{align}
	Because of $\widetilde x\in \mathcal{X}_{\textnormal{I}},$ application of Lemma \ref{lem.tech.general.lambda-xl} yields
	\begin{align*}
		\Delta_\lambda(\widetilde x) <
		\frac{\alpha}{1+\alpha}\Big(1+\frac{1-w}{w} g(\widetilde x)\Big) \;.
	\end{align*}
	Observe that since $\widetilde x\in \mathcal{X}_{\textnormal{I}},$ $\widetilde x> \lambda +R_1(\widetilde x) > R_1(\widetilde x).$ With Lemma \ref{lem.tech.g(R1)}, we have $g(\widetilde x)\leq g(R_1(\widetilde x)) =O(\alpha^\gamma).$ Together with \eqref{eq.sdf1} and using that $\tfrac 12 (1-\alpha) \alpha/(1+\alpha)= \alpha/2+O(\alpha^2),$ the lower bound in (i) follows since $\gamma \in (0,1].$

	(ii) To prove the lower bound in (ii), we derive a sharper lower bound for $\Delta_\lambda(\widetilde x),$ where $\widetilde{x} :=\inf \operatorname{L}_{\alpha}^{-1}(\theta_0)$ as before. Introduce $x^*:=\sup \operatorname{U}_{\alpha}^{-1}(\theta_0)$ and observe that by assumption, $x^* \geq \lambda.$ If $x^*\in \mathcal{X}_{\textnormal{I}},$ we have by Lemma \ref{lem.tech.R1} that $\theta_0 = x^*+R_1(x^*) \geq x^*+G^{-1}(1 - \tfrac{\alpha}{1+\alpha}(1 + \tfrac{1-w}{w}g(\lambda))).$ If $x^* \in \mathcal{X}_{\textnormal{II}} \cup \mathcal{X}_{\textnormal{III}},$ then, using $x^* \geq \lambda,$ we have by Theorem \ref{thm.altern_Ua} and the definition of the functions $R_2,R_3$ in \eqref{eq.Rsdef}, $\theta_0 = x^*+R_2(x^*)\wedge R_3(x^*) \geq x^* + A,$ with $A := G^{-1}(1-\alpha - \tfrac{1-w}w \alpha g(\lambda)).$ This shows that $\theta_0 \geq x^*+ A$ for all $x^* \geq \lambda.$ Recall that $\theta_0 > \lambda$ implies $\widetilde x \in \mathcal{X}_{\textnormal{I}}.$  Using Lemma \ref{lem.tech.R1} and Lemma \ref{lem.well_def}, we obtain $R_1(\widetilde{x}) > A \vee 0 \geq A/2.$ Combining the lower bounds on $\theta_0$ and $R_1(\widetilde{x})$ for $x^* \geq \lambda$, we have
	\begin{align*}
		\widetilde x &= \theta_0 + R_1(\widetilde x) \geq \lambda + A + \frac 12 A    := \lambda + \frac 32 G^{-1}\Big(1-\alpha - \frac{1-w}w \alpha g(\lambda)\Big).
	\end{align*}
	Using $\Delta_{\lambda}(x) = G(\lambda - x) - G(-\lambda - x) < G(\lambda - x)$, $-G^{-1}(p)=G^{-1}(1-p)$ and Assumption \ref{def.tail.decay}, the lower bound on $\widetilde{x}$ derived in the previous display gives
	\begin{align*}
		\Delta_\lambda(\widetilde x)
		&< G(\lambda- \widetilde x) \leq G\bigg(-\frac 32 G^{-1}\Big(1-\alpha - \frac{1-w}w \alpha g(\lambda)\Big)\bigg) \\
		&\leq c_* \bigg(\alpha + \frac{1-w}w \alpha g(\lambda)\bigg)^{1+\gamma} \\
		& =O(\alpha^{1+\gamma}).
	\end{align*}
	
	By arguing as in the proof of (i), replacing the bound on $\Delta_\lambda(\widetilde x)$ by $\Delta_\lambda(\widetilde x)=O(\alpha^{1+\gamma}),$ the conclusion of part (ii) follows.

	(iii) Write $\overline{x} := \sup \operatorname{L}_{\alpha}^{-1}(\theta_0).$ Recall that $\overline{x} \in \mathcal{X}_{\textnormal{I}}$ and $\theta_0=\overline{x}-R_1(\overline{x}).$ Rewriting this and applying the definition of the function $R_1$ in \eqref{eq.Rsdef}
	\begin{align}
		C^{-}(\theta_0) \leq  \mathbb{P}_{\theta_0}\big( X \in [\theta_0, \overline{x}]\big)
		= G\big(R_1(\overline{x})\big) - G(0)
		\leq \frac{1-\alpha}{2} \big(1- \Delta_\lambda(\overline{x}) \big).
		\label{eq.sdf2}
	\end{align}
	Using that $G$ is continuous, the identity $G(-G^{-1}(1-p))=p,$ the closed form of $R_1$ given in \eqref{eq.Rsdef}, and Assumption \ref{def.tail.decay},
	\begin{align*}
		\lim_{\theta_0 \downarrow \lambda} \Delta_\lambda(\overline{x}) 
		&= \lim_{\theta_0 \downarrow \lambda} G\big(\lambda-\theta_0-R_1(\overline{x}) \big) - 
		G\big(-\lambda-\theta_0-R_1(\overline{x}) \big)  \\
		&\geq  \lim_{\theta_0 \downarrow \lambda} G\big(-R_1(\overline{x}) \big) - 
		G(-2\lambda)  \\
		&\geq \frac{\alpha}{2} +\frac{1-\alpha}{2} \lim_{\theta_0 \downarrow \lambda} \Delta_\lambda(\overline{x})
		- c_*G(-\lambda)^{1+\gamma}.
	\end{align*}
	Since by assumption $G(-\lambda) \leq \alpha,$ rearranging the terms yields
	\begin{align*}
		\frac{1-\alpha}{2}\lim_{\theta_0 \downarrow \lambda} \Delta_\lambda(\overline{x}) 
		\geq \frac{\alpha (1-\alpha)}{2+2\alpha}- c_*\frac{\alpha^{1+\gamma}(1-\alpha)}{1+\alpha}
		= \frac{\alpha}{2}-O(\alpha^{1+\gamma}). 
	\end{align*}	
	Together with \eqref{eq.sdf2}, the claim follows.
\end{proof}

\begin{proof}[Proof of Proposition \ref{prop.Cpart2}]
	{\it (i):} We first show that if $G(-2\lambda) \leq \alpha/2,$ $\sup \mathcal{X}_{\textnormal{III}} < \lambda.$ By Theorem \ref{thm.altern_Ua} it is enough to show that $R_2(x)< R_3(x)$ for all $x\geq \lambda.$ Using the definitions of the functions $R_1$ and $R_2$ in \eqref{eq.Rsdef} and rewriting the expressions shows that $R_2(x)<R_3(x)$ if and only if $G(-\lambda-x) < \alpha/2+\tfrac{1-w}{2w}\alpha g(x) + \alpha \Delta_\lambda(x)/2.$ Since $G(-2\lambda)\leq \alpha/2$ implies the latter inequality for $x\geq \lambda,$ we must have that $\sup \mathcal{X}_{\textnormal{III}} < \lambda.$
	
	Set $\underline{x}:=\inf \operatorname{U}_{\alpha}^{-1}(\theta_0).$ By assumption $\underline{x} \geq \lambda,$ and therefore, $\underline{x}$ lies in regime $\mathcal{X}_{\textnormal{I}}$ or in regime $\mathcal{X}_{\textnormal{II}}.$ Suppose first that $\underline{x} \in \mathcal{X}_{\textnormal{I}}.$ Then, using Lemma \ref{lem.cov_reform}, the definition of the function $R_1$ in \eqref{eq.Rsdef} and $G(-G^{-1}(1-p))=p,$
	\begin{align}
		C^{+}(\theta_0)
		\leq G(0)-G\big( -R_1(\underline{x})\big) \leq \frac{1-\alpha}{2} .
		\label{eq.sdf3}
	\end{align}
	Similarly, for $\underline{x} \in \mathcal{X}_{\textnormal{II}},$ using also that $\Delta_\lambda(x) \leq G(\lambda-\underline{x})\leq G(0)=1/2$ and $G(-\lambda-\underline{x})\leq G(-2\lambda) \leq c_*\alpha^{1+\gamma},$
	\begin{align*}
		C^{+}(\theta_0)
		\leq G(0)-G\big( -R_2(\underline{x})\big) \leq \frac 12 - \alpha +\alpha \Delta_\lambda(\underline{x}) +G(-\lambda-\underline{x})\leq \frac{1-\alpha}{2}+O(\alpha^{1+\gamma}).
	\end{align*}
	Together with \eqref{eq.sdf3}, part (i) follows.
	
	{\it (ii):} Write $x^*:=\sup \operatorname{U}_{\alpha}^{-1}(\theta_0).$  By Lemma \ref{lem.basic_props_Ualpha} (i), $x^* \in \mathcal{X}_{\textnormal{I}}.$ Thus, 
	\begin{align*}
		C^{+}(\theta_0)
		\geq G(0)-G\big( -R_1(x^*)\big)
		= \frac{1-\alpha}{2}\big(1-\Delta_\lambda(x^*)\big) -\frac{1-w}{2w}\alpha g(x^*).
	\end{align*}
	Using $x^*\geq \lambda+\tfrac 32 G^{-1}(1-\alpha/2),$ $G(-G^{-1}(1-p))=p,$ and Assumption \ref{def.tail.decay}, we obtain $\Delta_\lambda(x^*)\leq G(\lambda-x^*) \leq G(-\tfrac 32 G^{-1}(1-\alpha/2))\leq c_*\alpha^{1+\gamma}.$ Since $x^* \in \mathcal{X}_{\textnormal{I}}$ implies $x^*>\lambda+R_1(x^*)\geq  R_1(x^*),$ Lemma \ref{lem.tech.g(R1)} yields $g(x^*)=O(\alpha^{1+\gamma}).$ This shows that $C^{+}(\theta_0) \geq (1-\alpha)/2-O(\alpha^{1+\gamma}),$ completing the proof for (ii).

	{\it (iii):} We first derive a lower bound. Again denote $x^*:=\sup \operatorname{U}_{\alpha}^{-1}(\theta_0)$ and recall that $x^* \geq \lambda >t_\alpha$ here. By Assumption \ref{def.tail.decay}, we have $g(\lambda) \leq c_*(1-G(\lambda))^{\gamma}=c_*G(-\lambda)^{\gamma}\leq c^* \alpha^{\gamma},$ where the latter follows from $G(-\lambda) \leq \alpha.$ If $x^* \in \mathcal{X}_{\textnormal{I}},$ we apply Lemma \ref{lem.cov_reform}. Combining Lemma \ref{lem.tech.R1} and $G(-G^{-1}(1-p))=p$ with the bound for $g(\lambda)$ then yields 
	\begin{align*}
		C^{+}(\theta_0)
		&\geq \frac 12 - G\big(-R_1(x^*)\big) \\
		&\geq \frac 12 -G\bigg( - G^{-1}\Big(1 - \frac{\alpha}{1+\alpha}\Big(1 + \frac{1-w}{w}g(\lambda)\Big)\Big) \bigg) \\
		&= \frac 12 - \frac{\alpha}{1+\alpha}\Big(1 + \frac{1-w}{w}g(\lambda)\Big) \\
		&\geq\frac 12 -\alpha - \frac{1-w}{w}c_* \alpha^{1+\gamma}.
	\end{align*}
	If $x^* \in \mathcal{X}_{\textnormal{II}},$ we obtain using the definition of the function $R_2$ in \eqref{eq.Rsdef} and $g(x^*) \leq g(\lambda),$
	\begin{align*}
		C^{+}(\theta_0)
		\geq \frac 12 - G\big(-R_2(x^*)\big)
		\geq \frac 12 - 	\alpha - \frac{1-w}{w}\alpha g(x^*) \geq\frac 12 -\alpha - \frac{1-w}{w}c_* \alpha^{1+\gamma}.
	\end{align*}
	With exactly the same argument, we also find $C^{+}(\theta_0)= \tfrac 12 -\tfrac \alpha 2 - \tfrac{1-w}{2w}c_* \alpha^{1+\gamma}$ if $x^*\in \mathcal{X}_{\textnormal{III}}$ and $x^* \geq \lambda.$ The lower bound follows by taking the infimum over $\{\theta_0: \sup \operatorname{U}_{\alpha}^{-1}(\theta_0) \geq \lambda \}.$ 
	
	We now derive an upper bound of $C^{+}(\theta_0)$ for 
	\begin{align*}
		\theta_0=\lambda+G^{-1}(1-\alpha)+G^{-1}\Big(1-\alpha + \frac{\alpha^2}{2}+\alpha G\Big(\frac 12 G^{-1}(\alpha) \Big)+G(-2\lambda)\Big).
	\end{align*}
	It will be enough to consider the case $\alpha \leq 1/2$ which implies that $G^{-1}(1-\alpha)\geq 0.$
	
	Let $\overline x = \lambda + \tfrac 12 G^{-1}(1-\alpha).$ We show that for any $x\in [\overline x, \lambda+G^{-1}(1-\alpha)],$ we have $U_\alpha(x)< \theta_0.$ By Lemma \ref{lem.mon}, $R_1$ is monotonically increasing on this interval and 
	\begin{align*}
		x+R_1(x) &\leq \lambda+G^{-1}(1-\alpha)+R_1\big(\lambda+G^{-1}(1-\alpha)\big) \\
		&< \lambda+G^{-1}(1-\alpha)+G^{-1}\Big(1-\alpha+\frac{\alpha^2}{2}+ G(-2\lambda)\Big) \\
		&\leq \theta_0.
	\end{align*}
	Since $R_2(x)<G^{-1}(1-\alpha +\alpha G(\lambda-x)+ G(-\lambda-x)),$ we also conclude that for any $x\in [\overline x, \lambda+G^{-1}(1-\alpha)],$
	\begin{align*}
		x+R_2(x) &< \lambda +G^{-1}(1-\alpha) + G^{-1}\big(1-\alpha +\alpha G(\lambda-\overline x) +G(-2\lambda)\big) \\
		&\leq 	\lambda +G^{-1}(1-\alpha) + G^{-1}\Big(1-\alpha +\alpha G\Big(-\frac 12 G^{-1}(1-\alpha)\Big)+G(-2\lambda) \Big) \\
		&\leq \theta_0.
	\end{align*}
	By Theorem \ref{thm.altern_Ua}, we obtain for $x\geq 0,$ $\operatorname{U}_{\alpha}(x) \leq x+(R_1(x)\vee R_2(x)).$ Combined with the bounds above, this proves $\operatorname{U}_{\alpha}(x) <\theta_0$ for all $x\in [\overline x, \lambda+G^{-1}(1-\alpha)].$
	
	It remains to be shown that this specific choice of $\theta_0$ satisfies the condition $\sup \operatorname{U}_{\alpha}^{-1}(\theta_0) \geq \lambda.$ Recall that $\theta_0 > \overline{x} > \lambda \geq t_{\alpha}$, and that $\operatorname{U}_{\alpha}$ is continuous on $(\lambda,\theta_0]$ by Theorem \ref{thm.altern_Ua}. Moreover, $\operatorname{U}_{\alpha}(\theta_0) \geq \theta_0 + (R_2(\theta_0) \vee R_1(\theta_0)) \geq \theta_0 + R_1(\theta_0) > \theta_0,$ where the latter inequality follows from Lemma \ref{lem.well_def}. Since it was just established that $\operatorname{U}_{\alpha}(\overline{x}) < \theta_0$, the intermediate value theorem ensures there exists an $x \in (\overline{x},\theta_0)$ s.t. $\operatorname{U}_{\alpha}(x) = \theta_0$. Because $x > \overline{x} > \lambda,$ the condition $\sup \operatorname{U}_{\alpha}^{-1}(\theta_0) \geq \lambda$ is satisfied.
	
	In a next step, we show that $\overline x \leq \theta_0-G^{-1}(1-c_*\alpha^{1+\gamma}).$ Inserting the definition of $\theta_0$ and $\overline x = \lambda + \tfrac 12 G^{-1}(1-\alpha),$ we need to verify that $G^{-1}(1-c_*\alpha^{1+\gamma}) \leq 3G^{-1}(1-\alpha)/2.$ By Assumption \ref{def.tail.decay} and using $G^{-1}(1-p)=-G(p),$
	\begin{align*}
		G\Big(-\frac{3}{2}G^{-1}(1-\alpha)\Big) \leq G\Big( \frac{3}{2}G^{-1}(\alpha) \Big) 
		\leq c_*\alpha^{1+\gamma},
	\end{align*}
	which by rearranging yields $G^{-1}(1-c_*\alpha^{1+\gamma}) \leq 3G^{-1}(1-\alpha)/2$ and therefore $\overline x \leq \theta_0-G^{-1}(1-c_*\alpha^{1+\gamma}).$
	
	Applying the results from the previous steps, we have that 
	\begin{align*}
		C^{+}(\theta_0) &\leq P_{\theta_0}\Big(X \in (-\infty, \overline x] \cup [\lambda+G^{-1}(1-\alpha), \theta_0] \Big) \\
		&\leq G\Big(-G^{-1}\big(1-c_*\alpha^{1+\gamma}\big)\Big) +\frac 12 - G\big( \lambda+G^{-1}(1-\alpha)-\theta_0\big) \\
		&= c_*\alpha^{1+\gamma}+ \frac 12 - G\bigg(-G^{-1}\Big(1-\alpha + \frac{\alpha^2}{2}+\alpha G\Big(\frac 12 G^{-1}(\alpha) \Big)+G(-2\lambda)\Big)\bigg)  \\
		&=\frac 12 -\alpha +O(\alpha^{1+\gamma})+\alpha G\Big(\frac 12 G^{-1}(\alpha) \Big)+G(-2\lambda).
	\end{align*}
	By Assumption \ref{def.tail.decay} and applying $G(-\lambda) \leq \alpha,$ it follows moreover that $G(-2\lambda) \leq c_*G(-\lambda)^{1+\gamma}\leq c_*\alpha^{1+\gamma}.$ This completes the proof for the upper bound.
	
	{\it (iv):} For $x\geq 0,$ the proof of Theorem \ref{thm.altern_Ua} shows that $\operatorname{U}_{\alpha}(x)=x+(R_2(x)\wedge R_3(x))\vee R_1(x),$ $\operatorname{U}_{\alpha}(x)> \lambda,$ and $\operatorname{U}_{\alpha}$ is a continuous function. Since $[0,2\lambda]$ is a compact interval, the minimum $A:=\min_{x\in [0,2\lambda]} \operatorname{U}_{\alpha}(x)$ is attained and we must have $A > \lambda.$ Set $\theta_0^*=A\wedge (2\lambda).$ For any $\theta_0\in (\lambda,\theta_0^*],$ it must hold that 
	\begin{align*}
		C^+(\theta_0)=P_{\theta_0}\big(\theta_0\in [X,\operatorname{U}_{\alpha}(X)]\big)\geq P_{\theta_0}\big(X\in [0,\theta_0]\big)
		= \frac{1}{2}-G(-\theta_0)\geq \frac{1}{2}-G(-\lambda).
	\end{align*}
	
	{\it (v):} Using \eqref{eq.coverage_upper} and $\lambda<\theta_0<t_\alpha,$
	\begin{align*}
		C^+(\theta_0)\leq P_{\theta_0}\big(\{\theta_0\geq X\} \cap \{|X|>t_\alpha\} \big)
		\leq P_{\theta_0}\big( X\leq -t_\alpha) =G(-t_\alpha-\theta_0)\leq G(-2\lambda).
	\end{align*}
\end{proof}

\begin{proof}[Proof of Theorem \ref{thm.main}]
	Observe that $\alpha \leq (4c_*)^{-1/\gamma}$ and $G(-2\lambda) \leq c_* \alpha^{1+\gamma}$ imply that $G(-2\lambda) \leq  c_*\alpha^{1+\gamma} \leq \alpha/4.$ The conditions of Proposition \ref{prop.upper_coverage} (ii) follow from the imposed assumptions. For Theorem \ref{thm.main} (i) it remains to show that $\sup \operatorname{U}_{\alpha}^{-1}(\theta_0) \geq \lambda+\tfrac 32 G^{-1}(1-\alpha/2)$ implies $\inf \operatorname{U}_{\alpha}^{-1}(\theta_0)\geq \lambda,$ such that Proposition \ref{prop.Cpart2} (i) and (ii) both hold as well. For that it will be enough to prove that $\sup_{x\leq \lambda}\operatorname{U}_{\alpha}(x) \leq \lambda+\tfrac 32 G^{-1}(1-\alpha/2).$ By Theorem \ref{thm.altern_Ua}, we find $\operatorname{U}_{\alpha}(x) \leq x+R_1(x)\vee R_2(x)$ whenever $\operatorname{U}_{\alpha}(x)>0.$ By Lemma \ref{lem.mon}, $R_1$ is monotone increasing and thus using $G(-2\lambda) \leq \alpha/4 < \alpha/2,$
	\begin{align*}
		\sup_{x\leq \lambda} \, x+R_1(x)
		&\leq \lambda+R_1(\lambda) \leq \lambda+G^{-1}\Big(1-\frac{\alpha}2-\frac{1-\alpha}{2}\Big[\frac 12-G(-2\lambda)\Big]\Big) \\ &\leq \lambda+G^{-1}\Big(\frac 34\Big).
	\end{align*}
	With $R_2(x)<G^{-1}(1-\alpha +\alpha G(\lambda-x)+ G(-\lambda-x))=G^{-1}(1-\alpha G(x-\lambda)+ G(-\lambda-x)),$
	\begin{align*}
		\sup_{x\leq \lambda} \, x+R_2(x) < \lambda+G^{-1}\Big(1-\frac{\alpha}2+ G(-2\lambda)\Big)
		\leq \lambda+\frac 32 G^{-1}\Big(1-\frac{\alpha}{2}\Big),
	\end{align*}
	where the last inequality follows from $\alpha \leq (4c_*)^{-1/\gamma}$ and $G(\frac 32 G^{-1}(\frac{\alpha}2))\leq c_*\alpha^{1+\gamma}\leq \tfrac{\alpha}2-G(-2\lambda).$
	
	Combining the last two displays and $\alpha\leq 1/2$ gives $\sup_{x\leq \lambda}\operatorname{U}_{\alpha}(x) \leq \lambda+\tfrac 32 G^{-1}(1-\alpha/2).$ Using that $C^{-}(\theta_0)+C^+(\theta_0) =C(\theta_0)$ completes the proof.
	
	The second part of the theorem is Proposition \ref{prop.upper_coverage} (ii) combined with Proposition \ref{prop.Cpart2} (iii).
\end{proof}

\begin{proof}[Proof of Theorem \ref{thm.coverage_inclusion}]
	We show that the credible set under the prior $\pi^{\operatorname{MS}}_{\lambda}$ is contained in the credible set under $\pi_{\lambda}$ for $X\geq 0.$ The posterior distribution under $\pi^{\operatorname{MS}}_{\lambda}$ is $\pi^{\operatorname{MS}}_{\lambda}(\theta|X)=g(\theta-X) \mathbf{1}(\theta>\lambda)/(1-\Delta^{\operatorname{MS}}_{\lambda}(X))$ with $\Delta^{\operatorname{MS}}_{\lambda}(X) := G(\lambda - X).$ Since $\Delta_{\lambda}^{\operatorname{MS}}(x) = \Delta_{\lambda}(x) + G(-\lambda - x) > \Delta_{\lambda}(x),$ we have $\pi_{\lambda}^{\operatorname{MS}}(\theta|X)>\pi_\lambda(\theta | x)$ for $\theta>\lambda.$ If $X$ is non-negative and $X$ is in regime $\mathcal{X}_{\textnormal{I}}$ or in regime $\mathcal{X}_{\textnormal{II}},$ then the HPD $[\operatorname{L}_{\alpha}(X),\operatorname{U}_{\alpha}(X)]$ under $\pi_\lambda$ is contained in $[\lambda, \infty).$
	Together with the definition of the HPD and $\pi_{\lambda}^{\operatorname{MS}}(\theta|X)>\pi_\lambda(\theta | x)$, the HPD $[\operatorname{L}_{\alpha}^{\operatorname{MS}}(X),\operatorname{U}_{\alpha}^{\operatorname{MS}}(X)]$ under $\pi_\lambda^{\operatorname{MS}}$ must be contained in $[\operatorname{L}_{\alpha}(X),\operatorname{U}_{\alpha}(X)].$ Therefore, we have that 
	\begin{align}
		\begin{split}
			&\mathbb{P}_{\theta_0}\big(\{\theta_0 \in [\operatorname{L}_{\alpha}(X),\operatorname{U}_{\alpha}(X)]\}\cap \{X\in \mathcal{X}_{\textnormal{I}} \cup \mathcal{X}_{\textnormal{II}} \}\cap\{X\geq 0\}\big) \\
			&\geq \mathbb{P}_{\theta_0}\big(\{\theta_0 \in [\operatorname{L}_{\alpha}^{\operatorname{MS}}(X),\operatorname{U}_{\alpha}^{\operatorname{MS}}(X)]\}\cap \{X\in \mathcal{X}_{\textnormal{I}} \cup \mathcal{X}_{\textnormal{II}} \}\cap\{X\geq 0\}\big).
		\end{split}\label{eq.678}	
	\end{align}
	Let us now study regime $\mathcal{X}_{\textnormal{III}}.$ Define
	\begin{align*}
		R_1^{\operatorname{MS}}(X)&:= G^{-1}\Big( \frac 12 +\frac{1-\alpha}{2}G(x-\lambda)\Big)=G^{-1}\Big(1-\frac{\alpha}{2}-\frac{1-\alpha}{2}G(\lambda-X)\Big), \\
		R_2^{\operatorname{MS}}(X)&:=G^{-1}\big(1-\alpha +\alpha G(\lambda-X)\big)=G^{-1}\big(1-\alpha G(X-\lambda)\big).
	\end{align*}
	Using Lemma 3 (a) in \cite{MS2008} for the first equality it follows that $	\operatorname{U}_{\alpha}^{\operatorname{MS}}(X)$ is given by
	\begin{align*}
		&(X+R_2^{\operatorname{MS}}(X))\mathbf{1}\Big(X\leq \lambda+G^{-1}\Big(\frac 1{1+\alpha}\Big)\Big)+(X+R_1^{\operatorname{MS}}(X))\mathbf{1}\Big(X> \lambda+G^{-1}\Big(\frac 1{1+\alpha}\Big)\Big)  \\
		&= X+R_1^{\operatorname{MS}}(X)\vee R_2^{\operatorname{MS}}(X).	
	\end{align*}
	For $X\geq 0,$ we have that $G(\lambda-X)+G(-\lambda-X)\leq G(\lambda)+G(-\lambda)=1$ and consequently also $R_3(X)\geq R_2^{\operatorname{MS}}(X)$ as long as $X\geq 0.$ Since $G(-\lambda-X) \leq G(\lambda-X),$ we have that $-\tfrac{1-\alpha}2G(\lambda-X)\leq 0\leq  \tfrac{\alpha}2(G(\lambda-X)-G(-\lambda-X))$ and therefore also $R_3(X)\geq R_1^{\operatorname{MS}}(X)$ for all $X.$ This proves $\operatorname{U}_{\alpha}(X)\geq \operatorname{U}_{\alpha}^{\operatorname{MS}}(X)$ if $X\in \mathcal{X}_{\textnormal{III}}$ and $X\geq 0.$ Since in regime $\mathcal{X}_{\textnormal{III}},$ $\operatorname{L}_{\alpha}(X)\leq -\lambda,$ we must have that $[\operatorname{L}_{\alpha}^{\operatorname{MS}}(X),\operatorname{U}_{\alpha}^{\operatorname{MS}}(X)]\subseteq [\operatorname{L}_{\alpha}(X),\operatorname{U}_{\alpha}(X)]$ if $X\in \mathcal{X}_{\textnormal{III}}$ and $X\geq 0.$ Thus, \eqref{eq.678} extends to 
	\begin{align*}
		C(\theta_0) &\geq \mathbb{P}_{\theta_0}\big(\{\theta_0 \in [\operatorname{L}_{\alpha}(X),\operatorname{U}_{\alpha}(X)]\}\cap \{X\geq 0\}\big) \\
		&\geq \mathbb{P}_{\theta_0}\big(\{\theta_0 \in [\operatorname{L}_{\alpha}^{\operatorname{MS}}(X),\operatorname{U}_{\alpha}^{\operatorname{MS}}(X)]\}\cap \{X\geq 0\}\big)
		\geq C^{\operatorname{MS}}(\theta_0)-\mathbb{P}_{\theta_0}(X<0).
	\end{align*}
\end{proof}

\bibliographystyle{acm}       
\bibliography{bibCoverage}

\end{document}